\documentclass[11pt]{amsart}
\usepackage{hyperref}
\usepackage{fullpage}
\usepackage{dsfont}
\usepackage{enumitem}
\usepackage{tikz,color}
\usetikzlibrary{shapes.misc}

 \newtheorem{theorem}{Theorem}[section]
 \newtheorem{corollary}[theorem]{Corollary}

 \newtheorem{lemma}[theorem]{Lemma}

 \newtheorem{proposition}[theorem]{Proposition}
 
 \theoremstyle{definition}

 \newtheorem{definition}[theorem]{Definition}
 
 \newtheorem{remark}[theorem]{Remark}

\definecolor{mygreen}{rgb}{0.1,0.8,0.1}
\definecolor{mygray}{rgb}{0.7,0.7,0.7}

\newcommand{\beq}{\begin{equation}}
\newcommand{\eeq}{\end{equation}}

\newcommand{\beqn}{\begin{eqnarray}}
\newcommand{\eeqn}{\end{eqnarray}}

\newcommand{\Config[1]}{\mathrm{Config}\left( #1 \right)}
\newcommand{\Stable[1]}{\mathrm{Stable}\left( #1 \right)}
\newcommand{\ASMRec[1]}{\mathrm{ASMRec}\left( #1 \right)}
\newcommand{\SSMRec[1]}{\mathrm{SSMRec}\left( #1 \right)}
\newcommand{\Rec[1]}{\mathrm{Rec}\left( #1 \right)}
\newcommand{\ASMMinRec[1]}{\mathrm{ASMMinRec}\left( #1 \right)}
\newcommand{\SSMMinRec[1]}{\mathrm{SSMMinRec}\left( #1 \right)}
\newcommand{\MinRec[1]}{\mathrm{MinRec}\left( #1 \right)}
\newcommand{\Sub[1]}{\mathrm{Sub}\left( #1 \right)}

\newcommand{\Confign}{\mathrm{Config}\left( W_n \right)}
\newcommand{\Stablen}{\mathrm{Stable}\left( W_n \right)}
\newcommand{\SSMRecn}{\mathrm{SSMRec}\left( W_n \right)}
\newcommand{\ASMRecn}{\mathrm{ASMRec}\left( W_n \right)}

\newcommand{\SSMMinRecn}{\mathrm{SSMMinRec}\left( W_n \right)}
\newcommand{\ASMMinRecn}{\mathrm{ASMMinRec}\left( W_n \right)}

\newcommand{\Subn}{\mathrm{Sub}\left( C_n \right)}

\newcommand{\Del[1]}{\mathrm{Del}_{#1}}
\newcommand{\DiffDel[1]}{\mathrm{DiffDel}_{#1}}

\newcommand{\Kimb[1]}{\mathrm{Kimb}_{#1}}

\newcommand{\level}{\mathrm{level}}
\newcommand{\w}{\mathrm{weight^{01^*}}}

\newcommand{\OO}{\mathcal{O}}
\newcommand{\OM}{\dot{\mathcal{O}}}
\newcommand{\In[1]}{\mathrm{in}^{\OO}_{#1}}
\newcommand{\Out[1]}{\mathrm{out}^{\OO}_{#1}}
\newcommand{\orighta}{\xrightarrow{\OO}}
\newcommand{\olefta}{\xleftarrow{\OO}}

\newcommand{\PMO}{\mathrm{PMO}\left(C_n \right)}
\newcommand{\PMW}{\mathrm{PMW}}

\newcommand{\dgr[1]}{\mathrm{deg}_{#1}}

\newcommand{\T}{\mathrm{Topp}}

\newcommand{\Z}{\mathbb{Z}}
\newcommand{\Zp}{\mathbb{Z}_+}
\newcommand{\N}{\mathbb{N}}

\newcommand{\latticegrid}{
\foreach \x in {0,1,2}
  \foreach \y in {0,1}
    \draw [fill=mygray, color=mygray] (\x,\y) circle [radius=0.1];
\draw [fill] (0,0) circle [radius=0.1];
\draw [fill] (2,1) circle [radius=0.1];
}

\begin{document}

\title{Combinatorial aspects of sandpile models on wheel and fan graphs}
\author{Thomas Selig}
\address{Department of Computing, School of Advanced Technology, Xi'an Jiaotong-Livepool University} 
\email{Thomas.Selig@xjtlu.edu.cn}
\date{\today}

\begin{abstract}
We study combinatorial aspects of the sandpile model on wheel and fan graphs, seeking bijective characterisations of the model's recurrent configurations on these families. 
For wheel graphs, we exhibit a bijection between these recurrent configurations and the set of subgraphs of the cycle graph which maps the level of the configuration to the number of edges of the subgraph. 
This bijection relies on two key ingredients. The first consists in considering a stochastic variant of the standard Abelian sandpile model (ASM), rather than the ASM itself. The second ingredient is a mapping from a given recurrent state to a canonical minimal recurrent state, exploiting similar ideas to previous studies of the ASM on complete bipartite graphs and Ferrers graphs. 
We also show that on the wheel graph with $2n$ vertices, the number of recurrent states with level $n$ is given by the first differences of the central Delannoy numbers. 
Finally, using similar tools, we exhibit a bijection between the set of recurrent configurations of the ASM on fan graphs and the set of subgraphs of the path graph containing the right-most vertex of the path. We show that these sets are also equinumerous with certain lattice paths, which we name Kimberling paths after the author of the corresponding entry in the Online Encyclopedia of Integer Sequences.
\end{abstract}

\maketitle

%


\section{Introduction}\label{sec:intro}

The sandpile model is a dynamic process whereby grains of sand move around on a graph. Informally (a formal introduction is given in Section~\ref{sec:prelim}), grains of sand are placed on the vertices of the graph (we call this a \emph{configuration}). 
At each step of time, a vertex $v$ is chosen at random, and a grain of sand is added to $v$. If this causes the number of grains at $v$ to exceed a certain threshold (typically the degree of $v$), $v$ is said to be unstable, and topples, sending grains of sand to its neighbours according to some toppling rules. 
This may cause other vertices to become unstable, and topple in turn. There is an exit point for the system called the sink vertex, which absorbs grains, and so the process eventually stabilises. 
The dynamics of the sandpile model are given by this ``add grain at random, and stabilise'' operation.

In the classical so-called Abelian sandpile model (ASM), the toppling rules referred to above are simple (and deterministic): when a vertex topples, it sends one grain to each of its neighbours in the graph. Later in this section we will discuss stochastic variants of this model. 
The ASM was originally introduced on 2-dimensional grids by Bak, Tang and Wiesenfeld in the 1980's~\cite{BTW1, BTW2} as a model exhibiting a property known as \emph{self-organised criticality}. This means that over time, a system will converge to a critical state without needing to fine-tune any of its external parameters. 
The model was then generalised by Dhar~\cite{Dhar1}, who showed the following Abelian property from which its name is derived. If $a_v$ denotes the operator ``add a grain of sand at vertex $v$ and stablilise'', then the operators $(a_v)$ commute, i.e. $a_v a_w = a_w a_v$ for all vertices $v,w$. 
In graph theory, the model is sometimes referred to as the \emph{chip-firing game} (see e.g.~\cite{BLS}). For a recent extensive review of known results on this model, we refer the reader to Klivans's excellent book on the subject~\cite{Kliv}.

Of central importance in ASM research are the so-called \emph{recurrent configurations}, those that appear infinitely often in the long-time running of the model (in chip-firing terminology, these are called \emph{critical configurations}). 
Dhar~\cite{Dhar1} showed that the set of recurrent configurations forms an Abelian group when equipped with the addition operation ``vertex-pointwise addition and stablisation''. He named this the sandpile group of the graph. 
This group, also known as the critical group, graph Jacobian, or Picard group, has been widely studied in algebraic graph and matroid theory. Its most celebrated contribution is perhaps its role in developing discrete Riemann-Roch and Abel-Jacobi theory on graphs~\cite{Baker}. 
The structure of the sandpile group has also been explicitly computed on a variety of graph families, including complete graphs~\cite{Cori}, nearly-complete graphs (complete graphs with a cycle removed)~\cite{Zhou}, wheel graphs~\cite{Biggs1} (see also~\cite{Raza2, Raza1}), Cayley graphs of the dihedral group $D_n$~\cite{DFF}, and many more (see e.g.~\cite{Chen} and the references therein for a more complete list).

In enumerative and bijective combinatorics, a fruitful direction of recent ASM research has also focused on graph families, but instead of looking at the sandpile group, the focus has been on calculating the set of recurrent configurations via bijections to other more easily enumerated or generated combinatorial objects. 
Examples of graph families and related combinatorial objects are:
\begin{itemize}[topsep=2pt, noitemsep]
\item complete graphs~\cite{Cori}, corresponding to parking functions;
\item complete multi-partite graphs with a dominating sink~\cite{CorPou}, corresponding to generalised parking functions (called $(p_1,\ldots,p_k)$-parking functions);
\item complete bipartite graphs where the sink is in one of the two components~\cite{DLB}, corresponding to (labelled) parallelogram polyominoes (see also~\cite{AD, AADB, ALB});
\item complete split graphs~\cite{Dukes}, correspondences to the so-called \emph{tiered} parking functions and Motzkin words;
\item Ferrers graphs~\cite{DSSS1, SSS}, corresponding to (decorated) EW-tableaux (see Section~\ref{subsec:minrec} for more details);
\item permutation graphs~\cite{DSSS2}, corresponding to tiered trees and non-ambiguous binary trees.
\end{itemize}

In this paper, we study combinatorial aspects (in the above sense) of the sandpile model on wheel and fan graphs. 
We show that the recurrent configurations on these graphs are related to a variety of combinatorial objects: subgraphs of the cycle or path graphs, marked words/orientations, and two families of lattice paths known as Delannoy paths and Kimberling paths. 
In the wheel graph case, to exhibit the desired correspondences, it is better to consider a stochastic variant of the standard ASM, where the toppling rules of unstable vertices are random. 
In this model, instead of sending a grain to each neighbour, an unstable vertex chooses a random subset of neighbours to send grains to, and keeps the remaining grains. 
There are a number of such stochastic sandpile models in the literature, although on the whole they have not been nearly as widely studied as the classical Abelian model.
\begin{itemize}
\item In~\cite{Manna}, there are two different types of grain, which cannot occupy the same vertex. When they do, one of the grains is moved to a randomly chosen neighbouring vertex instead.
\item In~\cite{DharSad}, unstable vertices lose all their grains, which are re-distributed at random to their neighbours: some neighbours may receive more than one grain, others none, while some grains may exit the system directly.
\item In ~\cite{CMS}, unstable vertices flip a coin for each neighbour to decide which neighbours to send grains to. That is, all neighbours independently of each other receive a grain with probability $p  \in (0,1)$ (with probability $(1-p)$ that grain is kept by the toppling vertex).
\item The model in~\cite{Nunzi} essentially generalises the two previous models in~\cite{CMS, DharSad}.
\item In~\cite{KW}, the toppling threshold of a vertex is set to a (fixed) multiple $M$ of its degree. For each toppling, a random number $\gamma \in \{1,\ldots,M\}$ is chosen, and each neighbour of the toppling vertex receives the same (random) number $\gamma$ of grains.
\end{itemize}
The model that proves to be the most useful for our purposes here is the ``coin-flipping'' model~\cite{CMS} (see Section~\ref{subsec:SSM_prelim} for a more formal definition of this model).

The paper is organised as follows. In Section~\ref{sec:prelim}, we formally introduce the sandpile models (Abelian and stochastic) on general graphs. 
We give some important results on the sandpile model, specifically on characterisations of its recurrent configurations. 
In Section~\ref{sec:sandpile_wheel} we focus on the sandpile model on wheel graphs. Our main result is exhibiting a bijection between the set of recurrent configurations and the set of subgraphs of the cycle graph (Theorem~\ref{thm:bij_rec_subgraph}). 
In Section~\ref{sec:delannoy} we focus on the specific case of wheel graphs with an even number of vertices, and show that the recurrent configurations whose total number of grains is one-and-a-half times the number of vertices are enumerated by the first differences of the so-called Delannoy numbers (Theorem~\ref{thm:delannoy}). 
In Section~\ref{sec:sandpile_fan}, we study the sandpile model on fan graphs. Our main results are that their recurrent configurations are in bijection with certain subgraphs of the path graph (Theorem~\ref{thm:rec_fan_subgraph_path}), and are equinumerous with certain lattice paths (Theorem~\ref{thm:rec_fan_kimb}), which we call Kimberling paths after the name of their author in the OEIS~\cite{OEIS}.
Finally, Section~\ref{sec:conc} summarises the main results of our papers, and gives some suggestions for future work.


\section{Preliminaries}\label{sec:prelim}

In this section, we introduce some of the necessary definitions, tools and notation we will need and use throughout the rest of the paper. As usual, the sets $\Z$ and $\N$ denote the sets of integers and (strictly) positive integers respectively. We let $\Zp:=\N \, \cup \, \{0\}$ denote the set of non-negative integers. For a positive integer $n \in \N$, we denote $[n]$, resp.\ $[n]_0$, the set $\{1,\ldots,n\}$, resp.\ $\{0,\ldots,n\}$.

Throughout this section, a \emph{graph} $G$ is a labelled, simple, undirected, connected graph with vertex set $V \subseteq [n]_0$ for some $n$. The edge set $E$ is therefore finite, and may not contain multiple edges or loops. 
For $i \in V$, we write $\dgr[i] = \dgr[i]^G$ for the degree of the vertex $i$ in $G$, omitting the superscript where the underlying graph is unambiguous. 
A \emph{subgraph} of $G$ is a pair $(A, E_A)$ where $A$ is a non-empty subset of $V$, and $E_A$ a subset of $E$ containing only (some of the) edges with both endpoints in $A$. We write $\Sub[G]$ for the set of subgraphs of a graph $G$.

The \emph{path} graph $P_n$ is the simple path graph with vertex set $[n]$, that is the edges are the pairs $\{i,i+1\}$ for $i \in [n-1]$. 
The \emph{cycle} graph $C_n$ is the simple cycle graph with vertex set $[n]$, obtained by adding an edge $\{n,1\}$ to $P_n$ to complete the cycle. 
The \emph{wheel} graph $W_n$ is constructed by taking the \emph{join} of the cycle graph $C_n$ and of a single isolated vertex $0$ (that is, the vertex $0$ is connected to all vertices in $C_n$). 
The \emph{fan} graph $F_n$ is constructed by removing the edge $\{n,1\}$ from $W_n$, or equivalently by taking the join of the path graph $P_n$ and of a single isolated vertex. 
We refer to vertex $0$ as the \emph{sink} of $W_n$ and $F_n$. Figure~\ref{fig:cycle_wheel_fan} shows an example of cycle, wheel and fan graphs, with the sink vertex represented as a square, and other vertices as circles. 

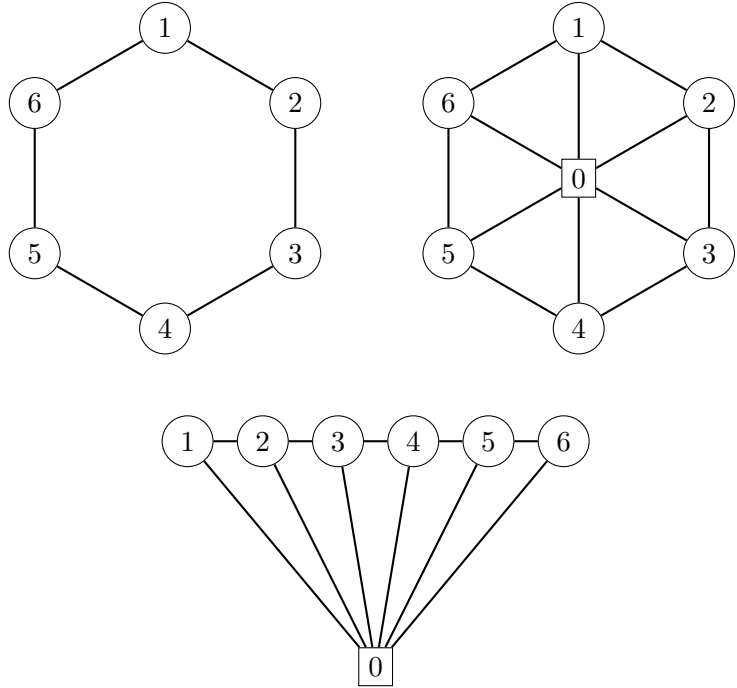
\begin{figure}[ht]
\centering
\begin{tikzpicture}

\node[circle, draw=black] (1) at (0,2) {$1$};
\node[circle, draw=black] (2) at ({sqrt(3)},1) {$2$};
\node[circle, draw=black] (3) at ({sqrt(3)},-1) {$3$};
\node[circle, draw=black] (4) at (0,-2) {$4$};
\node[circle, draw=black] (5) at (-{sqrt(3)},-1) {$5$};
\node[circle, draw=black] (6) at (-{sqrt(3)},1) {$6$};
\draw[thick] (1)--(2)--(3)--(4)--(5)--(6)--(1);

\begin{scope}[xshift=5.5cm]
\node[rectangle, draw=black] (0) at (0,0) {$0$};
\node[circle, draw=black] (1) at (0,2) {$1$};
\node[circle, draw=black] (2) at ({sqrt(3)},1) {$2$};
\node[circle, draw=black] (3) at ({sqrt(3)},-1) {$3$};
\node[circle, draw=black] (4) at (0,-2) {$4$};
\node[circle, draw=black] (5) at (-{sqrt(3)},-1) {$5$};
\node[circle, draw=black] (6) at (-{sqrt(3)},1) {$6$};
\draw[thick] (1)--(2)--(3)--(4)--(5)--(6)--(1);
\foreach \x in {1,...,6}
	\draw[thick] (0)--(\x);
\end{scope}

\begin{scope}[xshift=1.3cm,yshift=-5cm]
\node[rectangle, draw=black] (0) at (1.5,-1.5) {$0$};
\foreach \x in {1,...,6}
    \node[circle, draw=black] (\x) at (-2+\x,1.5) {$\x$};
\foreach \x in {1,...,6}
    \draw[thick] (0)--(\x);
\draw [thick] (1)--(2)--(3)--(4)--(5)--(6);
\end{scope}

\end{tikzpicture}

\caption{The cycle graph $C_6$ (top left), the wheel graph $W_6$ (top right), and the fan graph $F_6$ (bottom centre).
\label{fig:cycle_wheel_fan}
}

\end{figure}

An \emph{orientation} $\OO$ of the graph $G$ is the assignment of a direction to each edge in $E$. For an orientation $\OO$, and an edge $e=\{i,j\} \in E$, we write $i \orighta j$ to denote that the edge $e$ is directed from $i$ to $j$ in $\OO$. For $i \in V$, we denote $\In[i]$, resp.\ $\Out[i]$, the number of incoming, resp.\ outgoing, edges at $i$ in $\OO$.

An orientation is \emph{acyclic} if it contains no directed cycles. A \emph{target}, resp.\ \emph{source}, of an orientation is a vertex where all edges are incoming, resp.\ outgoing\footnote{The terminology ``sink'' is more widespread in the literature than ``target'', but the sandpile model already has a designated vertex called ``sink'', hence our choice of terminology here.}. It is straightforward to verify that an acyclic orientation contains at least one target and one source. For a given vertex $i \in V$, an orientation is said to be \emph{$i$-rooted} if the vertex $i$ is the orientation's unique target. 

\begin{remark}\label{rem:sink-rooted_or_wheel_fan}
A $0$-rooted orientation of the wheel graph $W_n$ can be mapped straightforwardly to an orientation of the cycle graph $C_n$ by removing the sink $0$ and its incident edges. This correspondence is clearly one-to-one. 
Similarly, we have a one-to-one correspondence between $0$-rooted orientations of the fan graph $F_n$ and orientations of the path graph $P_n$.
\end{remark}

\subsection{The Abelian sandpile model (ASM)}\label{subsec:ASM_prelim}

In this part we introduce the Abelian sandpile model (ASM) on a general graph $G$, and recall some important results regarding the so-called \emph{recurrent configurations} of the model. 
Let $G$ be a graph with vertex set $[n]_0$. As in the wheel and fan graph cases, we refer to $0$ as the \emph{sink} of the graph.

A \emph{configuration} on $G$ is a vector $c=(c_1,\ldots ,c_n) \in \Zp^n$ that assigns the number $c_i$ to vertex $i$. 
We think of $c_i$ as representing the number of grains of sand at the vertex $i$. 
Note that the sink vertex $0$ is not assigned a number of grains.
Denote by $\Config[G]$ the set of all configurations on $G$.
Let $\alpha_i \in \Zp^n$ be the configuration with $1$ in the $i$-th position and $0$ elsewhere. By convention, $\alpha_0$ is the all-0 configuration.

We say that a vertex $i$ in a configuration $c=(c_1,\ldots ,c_n)\in \Config[G]$ is \emph{stable} if $c_i < \dgr[i]$. 
Otherwise it is called \emph{unstable}. 
A configuration is called stable if all its (non-sink vertices) are stable, and we denote $\Stable[G]$ the set of all stable configurations on $G$.

Unstable vertices topple. 
We define the \emph{toppling operator} $\T_i$ corresponding to the toppling of an unstable vertex $i \in [n]$ in a configuration $c \in \Config[G]$ by:
\beq\label{eq:toppling}
\T_i(c) := c - \dgr[i] \alpha_i + \sum_{j: \{i,j\} \in E} \alpha_j,
\eeq
where the sum is over all vertices $j$ adjacent to $i$ in $G$, and the addition operator on configurations denotes pointwise addition at each vertex. In words, the toppling of a vertex $i$ sends one grain of sand to each neighbour of $i$ in $G$.

Performing this toppling may cause other vertices to become unstable, and we topple these also. 
One can show (see e.g.~\cite[Section~5.2]{Dhar}) that starting from some unstable configuration $c$ and toppling successively unstable vertices, we eventually reach a stable configuration $c'$ (think of the sink as absorbing grains). 
In addition, the configuration $c'$ reached does not depend on the sequence in which vertices are toppled. 
We call this $c'$ the \emph{stabilisation} of $c$.

We now define a Markov chain on the set $\Stable[G]$ of stable configurations.
Fix a probability distribution $\mu=(\mu_1,\dots,\mu_n)$ on $[n]$ such that $\mu_i>0$ for all $i \in [n]$.
At each step of the Markov chain we add a grain at the vertex $i$ with probability $\mu_i$ and stabilise the resulting configuration.

The \emph{recurrent} configurations are those appear infinitely often in the long-time running of this Markov chain. We let $\ASMRec[G]$ be the set of recurrent configurations for the ASM on the graph $G$. 
We use the notation $\ASMRec[G]$ is to mark the difference between configurations which are recurrent for the ASM and those for the SSM (stochastic sandpile model) which will be introduced in Section~\ref{subsec:SSM_prelim}.

The study of the recurrent configurations has been of central importance in ASM research. There are a number of ways to characterise them. Perhaps most famously Dhar's burning algorithm provides a straightforward algorithmic process to check if a given stable configuration is recurrent or not \cite[Section~6.2]{Dhar}. 
The result we make use of in this paper however is given in terms of so-called \emph{compatible} orientations. This result was first stated in these terms by Biggs \cite{Biggs2}, although the author credits a previous paper \cite{GZ} as having equivalent results.

Given a configuration $c = (c_1,\ldots,c_n) \in \Config[G]$ and an orientation $\OO$ of $G$, we say that $\OO$ and $c$ are \emph{compatible} if
\beq\label{eq:comp_orient_config}
\forall i \in [n], c_i \geq \In[i].
\eeq
If $\OO$ and $c$ are compatible, we will say that $\OO$ is compatible with $c$, or simply that $\OO$ is compatible if there is no ambiguity over which configuration $c$ is considered.

\begin{theorem}\label{thm:DR_AcOr}
Let $c = (c_1,\ldots,c_n) \in \Stable[G]$ be a \emph{stable} configuration on the graph $G$. Then $c$ is recurrent for the ASM if, and only if, there exists an \emph{acyclic, $0$-rooted} orientation $\OO$ compatible with $c$.
\end{theorem}

\subsection{A stochastic variant of the ASM}\label{subsec:SSM_prelim}

In this section, we introduce a stochastic variant of the ASM, called the stochastic sandpile model (SSM). We are more informal than in the previous section, as most of the notions introduced closely mirror those of the ASM. For a more formal definition of the SSM, see~\cite{CMS}. 

In the ASM, the only randomness in the model concerns the choice of which vertex to add a grain to at each step of the Markov chain (i.e. the distribution $\mu$ in Section~\ref{subsec:ASM_prelim}), while the subsequent operations -- topplings and thus stabilisation -- are deterministic.

In the SSM, we introduce an additional layer of randomness by making the topplings random. 
Fix some parameter $p \in (0,1)$. When we have an unstable vertex $i$, we flip a biased coin for each neighbour $j$ of $i$, and with probability $p$ one grain is sent from $i$ to $j$ (with probability $(1-p)$ the grain remains at $i$). 
The coin flips are independent for each neighbour $j$, and independent of other vertex topplings.

In~\cite[Theorem 2.2]{CMS} it is shown that, as in the ASM, starting from some unstable configuration $c$ and successively toppling unstable vertices, we eventually reach a stable configuration $c'$.
Moreover, this configuration $c'$ does not depend on the sequence in which vertices are toppled. 

As for the ASM, we can therefore define a Markov chain for the SSM on the set of stable configurations $\Stable[G]$, whose dynamics are given by:
\begin{enumerate}
\item Add a grain at a randomly chosen vertex according to some probability distribution $\mu=(\mu_1,\dots,\mu_n)$.
\item Stabilise the resulting configuration. The stabilisation operation here is random, according to the toppling rules described above.
\end{enumerate}
We denote the set of recurrent configurations for the SSM Markov chain $\SSMRec[G]$, to distinguish from the recurrent configurations of the standard ASM. The following was shown, in slightly different but equivalent form, in~\cite[Theorem 3.2]{CMS}.

\begin{theorem}\label{thm:SR_Or}
Let $c = (c_1,\ldots,c_n) \in \Stable[G]$ be a \emph{stable} configuration on the graph $G$. Then $c$ is recurrent for the SSM if, and only if, there exists a \emph{$0$-rooted} orientation $\OO$ compatible with $c$.
\end{theorem}

\begin{remark}\label{rem:diff_SR_DR}
The only difference between Theorems~\ref{thm:DR_AcOr} and \ref{thm:SR_Or} lies in the acyclicity (or not) of the compatible orientation. This immediately implies that $\ASMRec[G] \subseteq \SSMRec[G]$. 
In general, the converse doesn't hold. In Section~\ref{sec:sandpile_wheel} we will see that on the wheel graph $W_n$ there is exactly one additional recurrent configuration for the SSM, which is the configuration $(1,1,\ldots,1)$ with one grain of sand at each vertex.
\end{remark}

\subsection{A partial order on recurrent configurations}\label{subsec:partial_order_rec}

We now introduce a partial order on the set of recurrent configurations. Most of what we write here is valid for both the ASM and SSM, so we don't specify which model is being considered. 
We will simply denote by $\Rec[G]$ the set of recurrent configurations for the sandpile model on the graph $G$.

There is a natural partial order $\preceq$ on configurations of the sandpile model, defined as follows. If $c,c' \in \Config[G]$, we let $c \preceq c'$ if, and only if, $c_i \leq c'_i$ for all vertices $i \in [n]$. 
In other words $c'$ can be obtained from $c$ through a succession of grain additions. The following is an immediate consequence of Theorems~\ref{thm:DR_AcOr} and \ref{thm:SR_Or}.

\begin{proposition}\label{pro:po_rec}
Let $c \in \Rec[G]$ be a recurrent configuration, and $c' \in \Stable[G]$ a stable configuration satisfying $c \preceq c'$. Then $c'$ is recurrent.
\end{proposition}

We let $\MinRec[G]$ denote the set of minimal recurrent configurations on $G$ for the partial order $\preceq$, using the notation $\ASMMinRec[G]$ and $\SSMMinRec[G]$ when we need to specify that we are considering the ASM or SSM respectively. 
In words, a configuration is minimal recurrent if it is recurrent and removing a single grain from any vertex would cause it to no longer be so. These minimal recurrent configurations will play a crucial role in the rest of our paper.

\begin{theorem}\label{thm:MR_Or}
Let $c \in \Stable[G]$ be a stable configuration on $G$. Then $c$ is minimal recurrent for the ASM, resp.\ for the SSM, if, and only if, there exists an acyclic $0$-rooted, resp.\ $0$-rooted, orientation $\OO$ of $G$ such that:
\beq\label{eq:char_minrec_general_orientation}
\forall i \in [n], \, c_i = \In[i].
\eeq
Moreover, in the ASM case, such an orientation is unique.
\end{theorem}

Theorem~\ref{thm:MR_Or} was proved in the ASM case in~\cite{Schulz}. The proof in the SSM case is identical, taking into account the fact that the compatible orientation from Theorem~\ref{thm:SR_Or} is not necessarily acyclic. 
We also note that in the SSM case, the compatible orientation is unique up to \emph{cycle flipping}. That is, if $c$ is minimal recurrent for the SSM, and $\OO$ and $\OO'$ are two orientations satisfying Equation~\eqref{eq:char_minrec_general_orientation}, then $\OO'$ can be obtained from $\OO$ by successively flipping a finite number of directed cycles. This is a consequence of a general result linking out- or in-degree sequences of orientations and cycle flips (see e.g.~\cite[Lemma~40]{Ber} and references therein).

\subsection{The level statistic and level polynomial}\label{subsec:level}

Summing the Compatibility Inequality~\eqref{eq:comp_orient_config} gives the following inequality on the total number of grains of a recurrent configuration (for either the ASM or SSM):
$$ \sum\limits_{i \in [n]} c_i \geq \sum\limits_{i \in [n]} \In[i] = \vert E \vert - \dgr[0],$$
since in the right-hand sum every edge other than those incident to the sink $0$ is counted exactly once. This naturally leads to the following definition of the \emph{level} statistic for a recurrent configuration $c$:
\beq\label{eq:def_level}
\level(c) := \sum\limits_{i \in [n]} c_i + \dgr[0] - \vert E \vert,
\eeq
which is a non-negative integer. Note that the level of a configuration is essentially its total number of grains (up to an additive constant that only depends on the underlying graph), and that if $c$ is minimal recurrent, then $\level(c) = 0$.

The \emph{level polynomial} of a graph $G$ is then defined by:
\beq\label{eq:def_level_poly}
\mathrm{Level}_G(x) := \sum\limits_{c \in \Rec[G]} x^{\level(c)},
\eeq
where the sum is over all recurrent configurations on $G$. When we need to specify whether we are considering the ASM or SSM, we will write $\mathrm{ASMLevel}_G$ or $\mathrm{SSMLevel}_G$, as needed.

A well-known result links the ASM level polynomial to the ubiquitous Tutte polynomial of a graph.

\begin{theorem}\label{thm:Tutte_level}
Let $T_G(x,y)$ denote the bivariate Tutte polynomial of a graph $G$. We have:
$$ T_G(1,x) = \mathrm{ASMLevel}_G(x).$$
\end{theorem}

This result was first proved by Merino~\cite{Mer} following a conjecture by Biggs~\cite{Biggs2}. Since then, bijective proofs have been given in~\cite{Ber, CLB}, as well as in~\cite{DSSS2} in the specific case of permutation graphs. 
Note that the SSM level polynomial is known to satisfy a deletion-contraction relation that is very similar to that of the Tutte polynomial~\cite[Theorem~3.9]{CMS}.


\section{The sandpile model(s) on wheel graphs}\label{sec:sandpile_wheel}

In this section, we focus on the sandpile models (ASM and SSM) on wheel graphs. We first interpret Theorems~\ref{thm:DR_AcOr} and \ref{thm:SR_Or} in terms of the wheel graph $W_n$, giving straightforward characterisations of recurrent configurations. 
Then, we introduce one of the key ideas of our paper: the mapping from a recurrent configuration to a canonical minimal recurrent configuration. 
This allows us to exhibit a bijection from recurrent configurations of $W_n$ to \emph{properly-marked} orientations of $C_n$, which then yields the announced bijection to subgraphs of $C_n$.

Because of the cyclic symmetries inherent in the cycle and wheel graphs, it will be easier to mainly reason modulo $n$ on the vertex set $[n]$. 
For this, we should picture the graphs $C_n$ and $W_n$ as drawn on a plane, with vertices $1,2,\ldots,n$ in a clockwise cycle, as in Figure~\ref{fig:cycle_wheel_fan}. 
For two vertices $i,j \in [n]$, we write $[i,j]$ for the \emph{clockwise interval} from $i$ to $j$, that is $[i,j] := \{i, i+1, \ldots, j-1, j\}$ if $i \leq j$, and $[i,j] = \{i, i+1, \ldots, n\} \cup \{1, \ldots, j\}$ otherwise. We use similar notation for open and semi-open intervals.

\subsection{Straightforward characterisation of recurrent configurations on \texorpdfstring{$W_n$}{Wn}}\label{subsec:rec_wheel_easy}

In this part, we give some straightforward characterisations of recurrent configurations for the ASM and SSM on the wheel graphs $W_n$. 
Note that for the wheel graph $W_n$, we have $\Stablen = \{0,1,2\}^n$, i.e. the set of stable configurations is simply the set of words of length $n$ on $\{0,1,2\}$. 
Theorems~\ref{thm:DR_AcOr} and \ref{thm:SR_Or}, combined with Remark~\ref{rem:sink-rooted_or_wheel_fan}, immediately yield the following. 

\begin{proposition}\label{pro:wheel_rec_or}
Let $c = (c_1,\ldots,c_n) \in \{0,1,2\}^n$ be a \emph{stable} configuration on the wheel graph $W_n$. Then $c$ is recurrent for the SSM, resp.\ ASM, if, and only if, there exists an orientation, resp.\ acyclic orientation, of the cycle graph $C_n$, compatible with $c$.
\end{proposition}

The above proposition says that, on the wheel graph $W_n$, we can essentially ignore the sink in our study of the sandpile model. 
While this may seem straightforward, it is an important building block towards our main result. 
First, we show that recurrent configurations of the sandpile model on $W_n$ are given by words in $\{0,1,2\}^n$ satisfying simple conditions.

\begin{theorem}\label{thm:wheel_rec_words}
Let $c = (c_1,\ldots,c_n) \in \{0,1,2\}^n$ be a stable configuration on the wheel graph $W_n$. We have the following.
\begin{enumerate}
\item The configuration $c$ is recurrent for the SSM if, and only if, for all $i,j \in [n]$ such that $c_i = c_j = 0$, there exists $k \in (i,j)$ such that $c_k = 2$. By convention, if $i=j$, we let $(i,j) := [n] \setminus \{i\}$.
\item The configuration $c$ is recurrent for the ASM if, and only if, it is recurrent for the SSM and there exists $i \in [n]$ such that $c_i=2$.
\end{enumerate}
\end{theorem}

In words, Theorem~\ref{thm:wheel_rec_words} states that the recurrent configurations on $W_n$ are the words in $\{0,1,2\}^n$ such that between a pair of vertices with no grains of sand (cyclically), there must always be at least one vertex with two grains. We will refer to this as the $02$-cycle condition. 
The ASM recurrent configurations have the additional condition that there must be at least one vertex with two grains of sand. From this, we immediately deduce that:
\beq\label{eq:ssm_asm_rec_wheel}
\SSMRecn = \ASMRecn \cup \{ (1,1,\ldots,1) \},
\eeq
as announced in Remark~\ref{rem:diff_SR_DR}.
This result is actually noted in ~\cite{Cori} in the ASM case, with a proof relying on Dhar's burning algorithm. Since no such algorithm exists for the SSM, we give a full proof here using the Characterisation Proposition~\ref{pro:wheel_rec_or}. 
This proof also shows how to construct a compatible orientation from a recurrent configuration, which will be needed later in the paper.

\begin{proof}
We first show that if a configuration is recurrent, then the $02$-cycle condition holds. 
Let $c = (c_1, \ldots, c_n)$ be a recurrent configuration, and $i,j \in [n]$ such that $c_i = c_j = 0$. Since $c$ is recurrent, there exists an orientation $\OO$ of $C_n$ compatible with $c$ by Proposition~\ref{pro:wheel_rec_or}. 
The Compatibility Condition \eqref{eq:comp_orient_config} implies $\In[i] = \In[j] = 0$, i.e. $i$ and $j$ are sources of $\OO$. In particular, this implies that $i \orighta i+1$ (the edge is directed clockwise) and $j \orighta j-1$ (the edge is directed counter-clockwise). 
We can therefore define $k := \min \{k' \in (i,j); \, k'+1 \orighta k' \}$. By construction we have $k+1 \orighta k$ and $k-1 \orighta k$, i.e. $k$ is a target of the orientation $\OO$. 
By the Compatibility Condition \eqref{eq:comp_orient_config}, this implies that $c_k = 2$, as desired. Figure~\ref{fig:wheel_rec_words} provides an illustration of this part of the proof.

\begin{figure}[ht]

\centering

\begin{tikzpicture}

\node[circle, draw=black] (i) at ({sqrt(3)},1) {$i$};
\node[rounded rectangle, draw=black] (i+1) at (2,0) {$i+1$};
\node[rounded rectangle, draw=black] (k-1) at (1/2,-1.9) {$k-1$};
\node[circle, draw=black] (k) at (-1,-{sqrt(3)}) {$k$};
\node[rounded rectangle, draw=black] (k+1) at (-1.85,-3/4) {$k+1$};
\node[circle, draw=black] (j) at (-1/8,1.98) {$j$};
\node[rounded rectangle, draw=black] (j-1) at (-{sqrt(2)},{sqrt(2)}) {$j-1$};

\draw[thick,->] (i)--(i+1);
\draw[thick,->] (k-1)--(k);
\draw[thick,->] (k+1)--(k);
\draw[thick,->] (j)--(j-1);

\draw[thick,dashed,->] (i+1) to [out=-90, in=20] (k-1);
\draw[thick,dashed] (k+1) to [out=95, in=-110] (j-1);
\draw[thick,dashed] (j) to [out=0, in=120] (i);

\end{tikzpicture}

\caption{Illustration of the above proof. In the orientation $\OO$, all edges on the arc $[i,k]$ are directed clockwise, and $k$ is a target. Such a vertex must exist because $i \orighta i+1$ and $j \orighta j-1$. \label{fig:wheel_rec_words}}

\end{figure}
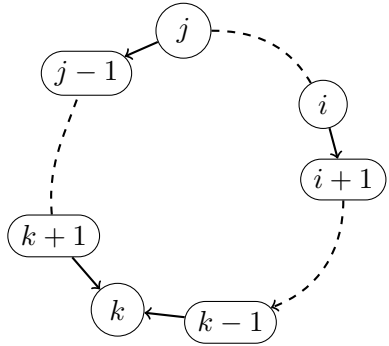

Now if $c$ is recurrent for the ASM, we know that in addition we can assume $\OO$ to be acyclic. Since an acyclic orientation must have at least one target say $i$, the Compatibility Condition~\eqref{eq:comp_orient_config} implies as above that $c_i = 2$, as desired.

\medskip

Let us now show the converse. Suppose that $c$ satisfies the $02$-cycle condition.  This implies that we either have $c = (1,\ldots,1)$ or there exists some $i \in [n]$ such that $c_i = 2$. If $c = (1, \ldots, 1)$ then the directed $n$-cycle orientation (in either direction) is clearly compatible with $c$, so that $c$ is recurrent for the SSM by Proposition~\ref{pro:wheel_rec_or}.

It remains to consider the case where $c_i = 2$ for some $i \in [n]$, and we will show that $c$ is recurrent for the ASM in this case (which implies that it is also recurrent for the SSM by Remark~\ref{rem:diff_SR_DR}). 
First, we show that we can assume that there exists $j \in [n]$ such that $c_j = 0$. Indeed, if this is not the case, we simply define $c'$ as $c'_{i-1} = 0$ and $c' = c$ elsewhere. 
The configuration $c'$ still satisfies the $02$-cycle condition, and since $c' \preceq c$, if we show that $c'$ is recurrent, this implies that $c$ is recurrent by Proposition~\ref{pro:po_rec}. 
Therefore, without loss of generality, we may assume that there exists $j \in [n]$ such that $c_j = 0$.

We now define an orientation $\OO$ of $C_n$ as follows. Set all vertices $j$ where $c_j = 0$ as sources of the orientation. 
Between two \emph{consecutive} sources $j < j'$, orient all edges clockwise from $j$ until we reach a first vertex $k$ with $c_k=2$ (such a vertex exists, since $c$ satisfies the $02$-cycle condition). 
Then orient all edges between $k$ and $j'$ counter-clockwise. By construction this orientation $\OO$ is compatible with $c$, and because it has at least one source it must be acyclic. It follows from Proposition~\ref{pro:wheel_rec_or} that $c$ is recurrent for the ASM, as desired.
\end{proof}

We now state a characterisation of \emph{minimal} recurrent configurations for the sandpile model on the wheel graph $W_n$.

\begin{theorem}\label{thm:wheel_min_rec_words}
Let $c = (c_1,\ldots,c_n) \in \{0,1,2\}^n$ be a stable configuration on the wheel graph $W_n$. Then $c$ is \emph{minimal} recurrent for the SSM if, and only if, for all $i,j \in [n]$ such that $c_i = c_j = 0$ and $c_{\ell} >0$ for all $\ell \in (i,j)$, there exists a \emph{unique} $k \in (i,j)$ such that $c_k = 2$. By convention, if $i=j$, we let $(i,j) := [n] \setminus \{i\}$. Moreover, $c$ is minimal recurrent for the ASM if, and only if, $c$ is minimal recurrent for the SSM and there exists $i \in [n]$ such that $c_i=2$.
\end{theorem}

Theorem~\ref{thm:wheel_min_rec_words} states that the minimal recurrent configurations on $W_n$ are the words in $\{0,1,2\}^n$ such that between a pair of \emph{consecutive} vertices with no grains of sand (cyclically), there must always be \emph{exactly one} vertex with two grains. 
Equivalently, this means that cyclically around the wheel, each vertex with no grains must be followed by a vertex with two grains, possibly with some vertices with one grain in between, and vice versa. We will refer to this as the \emph{strict} $02$-cycle condition. 
The (minimal) recurrent configurations for the ASM simply have the additional condition that there must be at least one vertex with two grains of sand, or equivalently, for the minimal recurrent configurations, at least one vertex with no grain. As in the case of general recurrent configurations, we have:
$$ \SSMMinRecn = \ASMMinRecn \cup \{ (1,1,\ldots,1) \}.$$

\begin{proof}
Suppose first that $c$ is minimal recurrent. Since $c$ is recurrent, it must satisfy the (general) $02$-cycle condition. But if there were two distinct vertices with two grains of sand between two consecutive vertices with no grains, removing one grain from either of those two vertices would not break the $02$-cycle condition, which contradicts the minimality of $c$. So $c$ must satisfy the strict $02$-cycle condition, as desired.

Conversely, suppose that $c$ satisfies the strict $02$-cycle condition. By Theorem~\ref{thm:wheel_rec_words}, $c$ is recurrent. But now, removing any grain from $c$ would break the $02$-cycle condition by producing a pair of vertices with no grain and no vertex with two grains in between. Thus, no grain can be removed from $c$ if the resulting configuration is to remain recurrent, which means exactly that $c$ is minimal recurrent.

The distinction between ASM and SSM comes directly from the results in the case of general recurrent configurations (Theorem~\ref{thm:wheel_rec_words}), so does not need to be considered here.
\end{proof}

The bijection in Theorem~\ref{thm:MR_Or} from a minimal recurrent configuration for the ASM on $W_n$ to its compatible orientation of $C_n$ is straightforward to describe. 
If a vertex has no grain of sand, the two edges on the cycle are outgoing. Working outwards from that vertex, all edges keep the same direction until we encounter a vertex with two grains, where the directions ``flip'' (the two edges incident to that vertex are incoming), and we then simply repeat the process. 
Figure~\ref{fig:config_or} shows an example of this construction. We write $\OO(c)$ for the orientation corresponding to a minimal recurrent configuration $c$.
For the additional minimal recurrent configuration $c = (1,\ldots,1)$ of the SSM, we have the choice between the two directed cycles of $C_n$ (clockwise or counter-clockwise). We fix, conventionally, $\OO(c)$ to be the counter-clockwise directed cycle in this case.

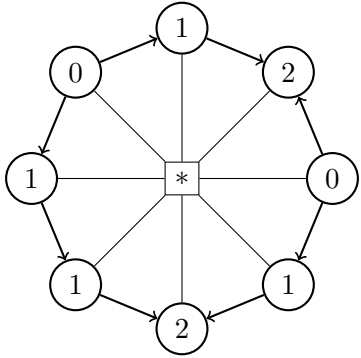
\begin{figure}[ht]
\centering
\begin{tikzpicture}

\node[rectangle, draw=black] (0) at (0,0) {$*$};
\node[circle, thick, draw=black] (1) at (0,2) {$1$};
\node[circle, thick, draw=black] (2) at ({sqrt(2)},{sqrt(2)}) {$2$};
\node[circle, thick, draw=black] (3) at (2,0) {$0$};
\node[circle, thick, draw=black] (4) at ({sqrt(2)},-{sqrt(2)}) {$1$};
\node[circle, thick, draw=black] (5) at (0,-2) {$2$};
\node[circle, thick, draw=black] (6) at (-{sqrt(2)},-{sqrt(2)}) {$1$};
\node[circle, thick, draw=black] (7) at (-2,0) {$1$};
\node[circle, thick, draw=black] (8) at (-{sqrt(2)},{sqrt(2)}) {$0$};
\draw[thick,->] (1)--(2);
\draw[thick,->] (3)--(2);
\draw[thick,->] (3)--(4);
\draw[thick,->] (4)--(5);
\draw[thick,->] (6)--(5);
\draw[thick,->] (7)--(6);
\draw[thick,->] (8)--(7);
\draw[thick,->] (8)--(1);
\foreach \x in {1,...,8}
	\draw (0)--(\x);

\end{tikzpicture}

\caption{A minimal recurrent configuration on $W_8$ and the corresponding compatible orientation of $C_8$; the number of incoming edges at each vertex is equal to the number of grains.
\label{fig:config_or}
}

\end{figure}

\subsection{Canonical minimal recurrent configurations}\label{subsec:minrec}

In this part, we introduce one of the key ideas of the paper, namely mapping each recurrent configuration of the SSM on $W_n$ to a corresponding canonical minimal recurrent configuration. 
This idea of mapping recurrent configurations to minimal recurrent configurations has been used in previous combinatorial sandpile research. 
In~\cite{DLB} the authors use it implicitly to define the ``bounce path'' in the parallelogram polyomino that a recurrent configuration on the complete bipartite graph is mapped to. 
The bounce path defines a minimal recurrent configuration (or equivalently, a ribbon polyomino), and this bounce path is then used to define a $(q,t)$-Narayana polynomial which enumerates recurrent configurations of the ASM on complete bipartite graphs according to certain parameters. 
In~\cite{DSSS1}, the authors first establish a bijection between minimal recurrent configurations for the ASM on Ferrers graphs and so-called EW-tableaux. 
A general recurrent configuration $c$ is then decomposed into a pair $(m(c), e(c))$ where $m(c)$ is the canonical minimal recurrent configuration associated with $c$, and $e(c)$ describes where grains need to be added to this $m(c)$ to re-constitute the original configuration $c$. 
This leads to a bijection between all recurrent configurations for the ASM on Ferrers graphs and decorated EW-tableaux.

In the wheel graph case, the construction of the canonical minimal recurrent configuration is very similar to the construction in the last part of the proof of Theorem~\ref{thm:wheel_rec_words}.

\begin{definition}\label{def:cyclic_first_max_weight}
Let $c \in \SSMRecn$ be a recurrent configuration for the SSM on the wheel graph $W_n$. We say that $i \in [n]$ is a \emph{cyclically first maximal vertex} if $c_i = 2$, and there exists $j \in [n]$ such that $c_j=0$ and for all $k \in (j,i)$ we have $c_k=1$.\\
A (clockwise) \emph{$01^*$-chain} of $c$ is a set of consecutive vertices $j, j+1, \ldots, j + k$ such that $c_j = 0$ and $c_{j'} = 1$ for $j+1 \leq j' \leq j+k$. The \emph{$01^*$-weight} of $c$, denoted $\w(c)$, is the total number of vertices belonging to a $01^*$-chain.
\end{definition}

In words, a cyclically first maximal vertex is a vertex with two grains of sand which is the first such vertex following some vertex with no grain (in a clockwise direction). 
The maximal $01^*$-chains of $c$ are the semi-open intervals from a vertex with no grains (included) to the following cyclically first maximal vertex (excluded). 
Then the $01*$-weight of $c$ is simply the sum of the lengths of these maximal chains. This is illustrated in Figure~\ref{fig:cyclic_first_max} below. 

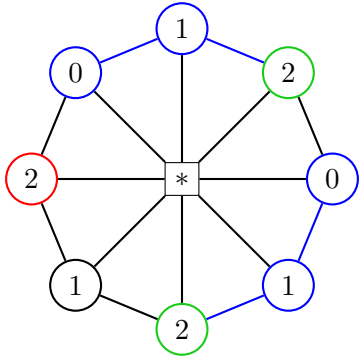
\begin{figure}[ht]
\centering
\begin{tikzpicture}

\node[rectangle, draw=black] (0) at (0,0) {$*$};
\node[circle, thick, draw=blue] (1) at (0,2) {$1$};
\node[circle, thick, draw=mygreen] (2) at ({sqrt(2)},{sqrt(2)}) {$2$};
\node[circle, thick, draw=blue] (3) at (2,0) {$0$};
\node[circle, thick, draw=blue] (4) at ({sqrt(2)},-{sqrt(2)}) {$1$};
\node[circle, thick, draw=mygreen] (5) at (0,-2) {$2$};
\node[circle, thick, draw=black] (6) at (-{sqrt(2)},-{sqrt(2)}) {$1$};
\node[circle, thick, draw=red] (7) at (-2,0) {$2$};
\node[circle, thick, draw=blue] (8) at (-{sqrt(2)},{sqrt(2)}) {$0$};
\draw[thick] (2)--(3);
\draw[thick] (5)--(6)--(7)--(8);
\draw[thick, blue] (8)--(1)--(2);
\draw[thick, blue] (3)--(4)--(5);
\foreach \x in {1,...,8}
	\draw[thick] (0)--(\x);

\end{tikzpicture}

\caption{A recurrent configuration on $W_8$ with its cyclically first maximal vertices in green, and its non cyclically first maximal vertices in red. The maximal $01^*$-chains are represented in blue, so this configuration has $01^*$-weight equal to $4$.
\label{fig:cyclic_first_max}
}

\end{figure}

We are now ready to define the canonical mapping from recurrent configurations to minimal recurrent configurations.

\begin{definition}\label{def:can_min_rec}
Let $c \in \SSMRecn$ be a recurrent configuration for the SSM on $W_n$. We define a configuration $c' := m(c) \in \Confign$ as follows:
$$ c'_i = \begin{cases}
0 & \text{if } c_i = 0 \\
2 & \text{if } i \mbox{ is a cyclically first maximal vertex } \\
1 & \text{otherwise}.
\end{cases}$$
\end{definition}

\begin{proposition}\label{pro:can_min_rec}
The map $m : c \mapsto m(c)$ given by Definition~\ref{def:can_min_rec} defines a surjection from $\SSMRecn$ to $\SSMMinRecn$, which reduces to the identity on $\SSMMinRecn$. 
Moreover, we have $\level(c) = \vert c \vert - \vert m(c) \vert$, where $\vert c \vert := \sum\limits_{i=1}^n c_i$ is the total number of grains of a configuration $c$. 
We call $m(c)$ the \emph{canonical minimal recurrent configuration} associated to $c$.
\end{proposition}

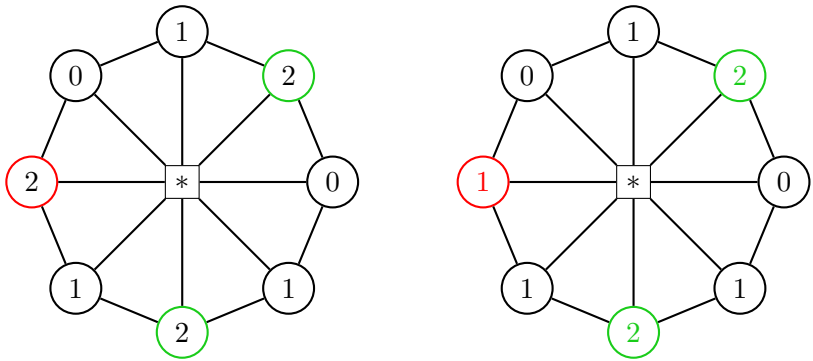
\begin{figure}[ht]
\centering
\begin{tikzpicture}

\node[rectangle, draw=black] (0) at (0,0) {$*$};
\node[circle, thick, draw=black] (1) at (0,2) {$1$};
\node[circle, thick, draw=mygreen] (2) at ({sqrt(2)},{sqrt(2)}) {$2$};
\node[circle, thick, draw=black] (3) at (2,0) {$0$};
\node[circle, thick, draw=black] (4) at ({sqrt(2)},-{sqrt(2)}) {$1$};
\node[circle, thick, draw=mygreen] (5) at (0,-2) {$2$};
\node[circle, thick, draw=black] (6) at (-{sqrt(2)},-{sqrt(2)}) {$1$};
\node[circle, thick, draw=red] (7) at (-2,0) {$2$};
\node[circle, thick, draw=black] (8) at (-{sqrt(2)},{sqrt(2)}) {$0$};
\draw[thick] (1)--(2)--(3)--(4)--(5)--(6)--(7)--(8)--(1);
\foreach \x in {1,...,8}
	\draw[thick] (0)--(\x);
	
\begin{scope}[xshift=6cm]
\node[rectangle, draw=black] (0) at (0,0) {$*$};
\node[circle, thick, draw=black] (1) at (0,2) {$1$};
\node[circle, thick, draw=mygreen, color=mygreen] (2) at ({sqrt(2)},{sqrt(2)}) {$2$};
\node[circle, thick, draw=black] (3) at (2,0) {$0$};
\node[circle, thick, draw=black] (4) at ({sqrt(2)},-{sqrt(2)}) {$1$};
\node[circle, thick, draw=mygreen, color=mygreen] (5) at (0,-2) {$2$};
\node[circle, thick, draw=black] (6) at (-{sqrt(2)},-{sqrt(2)}) {$1$};
\node[circle, thick, draw=red, color=red] (7) at (-2,0) {$1$};
\node[circle, thick, draw=black] (8) at (-{sqrt(2)},{sqrt(2)}) {$0$};
\draw[thick] (1)--(2)--(3)--(4)--(5)--(6)--(7)--(8)--(1);
\foreach \x in {1,...,8}
	\draw[thick] (0)--(\x);
\end{scope}

\end{tikzpicture}

\caption{Illustrating the map from a recurrent configuration (left) to its canonical minimal recurrent configuration (right).
\label{fig:minrec_map}
}

\end{figure}

\begin{proof}
By Theorem~\ref{thm:wheel_min_rec_words} all vertices with two grains of sand in a minimal recurrent configuration are cyclically first maximal vertices.  This implies that $m(c) = c$ if $c \in \SSMMinRecn$. 
Moreover, if $c \in \SSMRecn$, then it satisfies the general $02$-cycle condition. By construction, this implies that $m(c)$ satisfies the strict $02$-cycle condition, and is therefore minimal recurrent as desired. 
The formula for the level follows from the fact that since $m(c)$ is minimal recurrent, its level is $0$, and thus $\level(c) = \level(c) - \level(m(c)) = \vert c \vert - \vert m(c) \vert$.
\end{proof}

\subsection{Properly-marked orientations of \texorpdfstring{$C_n$}{Cn}}\label{subsec:marked_or}

In this part, we introduce the concept of properly-marked orientations of the cycle graph $C_n$, and exhibit a bijection between these objects and the set of recurrent configurations for the SSM on $W_n$. 
A marked orientation of a graph $G=(V,E)$ is simply a pair $\OM = (\OO, M)$ where $\OO$ is an orientation of $G$ and $M$ a subset of the vertex set $V$. We refer to the vertices of $M$ as being the \emph{marked} vertices of $\OM$.

\begin{definition}\label{def:pm}
A \emph{properly-marked} orientation of the cycle graph $C_n$ is a marked orientation $\OM = (\OO,M)$ of $C_n$ such that for each marked vertex $i \in M$, the two edges incident to $i$ are both oriented in a counter-clockwise direction in $\OO$ (that is, $i+1 \orighta i \orighta i-1$), and $\OO$ contains at least one counter-clockwise edge. The set of properly-marked orientations of $C_n$ is denoted $\PMO$.
\end{definition}

\begin{theorem}\label{thm:bij_rec_marked_or}
Let $c \in \SSMRecn$ be a recurrent configuration for the SSM on $W_n$. Define a marked orientation $\Phi_W(c) := \OM = (\OO, M)$ of $C_n$ by:
\begin{enumerate}
\item $\OO := \OO(m(c))$, where $m(c)$ is the canonical minimal recurrent configuration associated with $c$ from Definition~\ref{def:can_min_rec}, and $\OO(m(c))$ is the corresponding orientation of $C_n$, as exhibited at the end of Section~\ref{subsec:rec_wheel_easy};
\item $M$ is the set of vertices $i \in [n]$ such that $c_i = 2$ but $i$ is \emph{not} a cyclically first maximal vertex.
\end{enumerate}
Then $\Phi_W: \SSMRecn \rightarrow \PMO$ defines a bijection from the set of recurrent configurations for the SSM on $W_n$ to the set of properly-marked orientations of $C_n$. 
Moreover, for any configuration $c \in \SSMRecn$, we have $\level(c) = \vert M \vert$ and $\w(c) = \big\vert \{ \text{clockwise edges of } \OO \} \big\vert$.
\end{theorem}

\begin{remark}\label{rem:excess_grains}
In this construction, the marked vertices of the orientation are given by the vertices which are maximal (i.e. have two grains of sand) but not cyclically first maximal, in the corresponding recurrent configuration $c$. 
We can interpret these also as the locations of ``excess grains'' in $c$, that is the vertices where (exactly) one grain is removed from $c$ to obtain its canonical minimal recurrent configuration $m(c)$.
\end{remark}

\begin{proof}
There is in fact little work remaining to be done in the proof, since nearly all has been done through the constructions of Sections~\ref{subsec:rec_wheel_easy} and \ref{subsec:minrec}. 
We simply exhibit the inverse bijection $\Psi_W: \PMO \rightarrow \SSMRecn$ as follows. Let $\OM = (\OO, W)$ be a properly-marked orientation of $C_n$, and define $\tilde{c}$ to be the minimal recurrent configuration corresponding to the orientation $\OO$, i.e. $\tilde{c}_i := \In[i]$ for all $i \in [n]$. 
Then we define $c = \Psi_W(\OM) := \tilde{c} + \mathds{1}_M$ (i.e. $c$ is obtained from $\tilde{c}$ by adding one grain of sand to each marked vertex in $\OM$). It is clear that by construction $\Phi_W$ and $\Psi_W$ are inverses of each other, and the bijection is proved. 

The formula for the level follows immediately from the level formula in Proposition~\ref{pro:can_min_rec} and the observation that each marked vertex corresponds to exactly one grain of sand being removed in the transformation from $c$ to $m(c)$ (see Remark~\ref{rem:excess_grains}). 
The formula for the $01^*$-weight follows from the observation that each vertex $i$ in a $01^*$-chain corresponds bijectively to a clockwise edge $i \orighta i+1$. 
Figure~\ref{fig:bij_rec_pmor} provides an illustration of the different steps in the construction of $\Phi_W$.

\begin{figure}[ht]
\centering
\begin{tikzpicture}[scale=0.8]

\node[rectangle, draw=black] (0) at (0,0) {$*$};
\node[circle, thick, draw=blue] (1) at (0,2) {$1$};
\node[circle, thick, draw=black] (2) at ({sqrt(2)},{sqrt(2)}) {$2$};
\node[circle, thick, draw=blue] (3) at (2,0) {$0$};
\node[circle, thick, draw=blue] (4) at ({sqrt(2)},-{sqrt(2)}) {$1$};
\node[circle, thick, draw=black] (5) at (0,-2) {$2$};
\node[circle, thick, draw=black] (6) at (-{sqrt(2)},-{sqrt(2)}) {$1$};
\node[circle, thick, draw=red] (7) at (-2,0) {$2$};
\node[circle, thick, draw=blue] (8) at (-{sqrt(2)},{sqrt(2)}) {$0$};
\draw[thick] (1)--(2)--(3)--(4)--(5)--(6)--(7)--(8)--(1);
\foreach \x in {1,...,8}
	\draw[thick] (0)--(\x);
	
\begin{scope}[xshift=5.5cm]
\node[rectangle, draw=black] (0) at (0,0) {$*$};
\node[circle, thick, draw=blue] (1) at (0,2) {$1$};
\node[circle, thick, draw=black] (2) at ({sqrt(2)},{sqrt(2)}) {$2$};
\node[circle, thick, draw=blue] (3) at (2,0) {$0$};
\node[circle, thick, draw=blue] (4) at ({sqrt(2)},-{sqrt(2)}) {$1$};
\node[circle, thick, draw=black] (5) at (0,-2) {$2$};
\node[circle, thick, draw=black] (6) at (-{sqrt(2)},-{sqrt(2)}) {$1$};
\node[circle, thick, draw=red, color=red] (7) at (-2,0) {$1$};
\node[circle, thick, draw=blue] (8) at (-{sqrt(2)},{sqrt(2)}) {$0$};
\draw[thick] (1)--(2)--(3)--(4)--(5)--(6)--(7)--(8)--(1);
\foreach \x in {1,...,8}
	\draw[thick] (0)--(\x);
\end{scope}

\begin{scope}[xshift=11cm]
\node[circle, thick, draw=black] (1) at (0,2) {};
\node[circle, thick, draw=black] (2) at ({sqrt(2)},{sqrt(2)}) {};
\node[circle, thick, draw=black] (3) at (2,0) {};
\node[circle, thick, draw=black] (4) at ({sqrt(2)},-{sqrt(2)}) {};
\node[circle, thick, draw=black] (5) at (0,-2) {};
\node[circle, thick, draw=black] (6) at (-{sqrt(2)},-{sqrt(2)}) {};
\node[circle, thick, draw=red, color=red] (7) at (-2,0) {};
\node[circle, thick, draw=black] (8) at (-{sqrt(2)},{sqrt(2)}) {};
\draw[thick,->,blue] (1)--(2);
\draw[thick,->] (3)--(2);
\draw[thick,->,blue] (3)--(4);
\draw[thick,->,blue] (4)--(5);
\draw[thick,->] (6)--(5);
\draw[thick,->] (7)--(6);
\draw[thick,->] (8)--(7);
\draw[thick,->,blue] (8)--(1);
\end{scope}

\end{tikzpicture}

\caption{Illustrating the different steps in the construction of the bijection $\Phi_W$. On the left, a recurrent configuration $c$ on $W_8$ with its non cyclically first maximal vertex $v$ marked in red. In the middle, we have the canonical minimal recurrent configuration corresponding to $c$, keeping the marking on $v$. And on the right, the corresponding properly marked orientation of $C_8$ (with $M = \{v\}$ the previously marked vertex). The vertices on the $01^*$-chains of $c$ are in 1-1 correspondence with the clockwise-directed edges which immediately follow them (in blue).
\label{fig:bij_rec_pmor}
}

\end{figure}
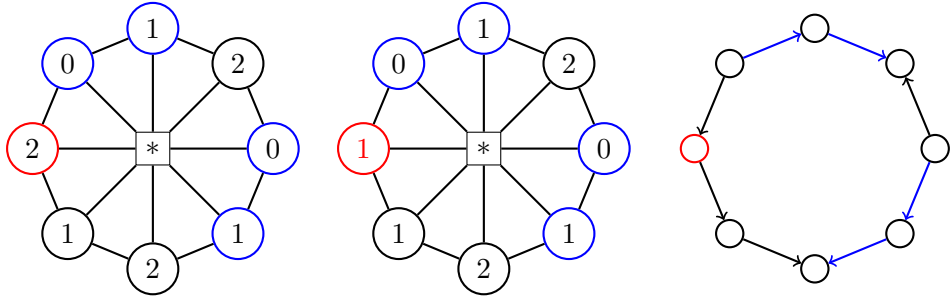

\end{proof}

\subsection{Subgraphs of \texorpdfstring{$C_n$}{Cn}}

We are now equipped to establish the main result of this section: a bijection between recurrent configurations for the SSM on $W_n$ and subgraphs of $C_n$ which maps the level, resp.\ $01^*$-weight, of a configuration to the number of edges in, resp.\ vertices not in, the corresponding subgraph. 
First, consider the following construction. Given a properly-marked orientation $\OM = (\OO, M)$ of $C_n$, we define a subgraph $S(\OM)$ of the \emph{edge-to-vertex dual} graph $C_n'$ by letting the vertices of $S(\OM)$ be those that correspond to edges directed counter-clockwise in $\OO$, and the edges of $S(\OM)$ be those that correspond to marked vertices in $M$. Since $\OO$ contains at least one counter-clockwise edge (Definition~\ref{def:pm}), by construction the subgraph $S(\OM)$ is non-empty (it contains at least one vertex).
This construction is illustrated on the right-hand side of Figure~\ref{fig:bij_rec_subgraph}.

By Definition~\ref{def:pm}, both edges incident to a marked vertex must be directed counter-clockwise. 
This means that the two end-points of the corresponding edge of $C_n'$ are present in $S(\OM)$, and therefore that $S(\OM)$ is indeed a subgraph of $C_n'$. 
It is then straightforward to check that $S$ defines a bijection from the set of properly-marked orientations of the cycle graph $C_n$ to the set of subgraphs of $C_n'$. 
Moreover, by construction, the number of marked vertices in $\OM$ is equal to the number of edges in $S(\OM)$, and the number of clockwise edges in $\OM$ is equal to the number of vertices \emph{not} in $S(\OM)$. 
Since the (edge-to-vertex) dual graph $C_n'$ is isomorphic to $C_n$, with a slight abuse of notation we get the following, which is the first main result of our paper.

\begin{theorem}\label{thm:bij_rec_subgraph}
The map $S \circ \Phi_W : \SSMRecn \rightarrow \Subn$ defines a bijection from the set of recurrent configurations for the SSM on the wheel graph $W_n$ to the set of subgraphs of the cycle graph $C_n$. 
Moreover, this bijection maps the level of a recurrent configuration $c$ to the number of edges in the corresponding subgraph $G$, and the $01^*$-weight of $c$ to the number of vertices that are \emph{not} in $G$.
\end{theorem}

Figure~\ref{fig:bij_rec_subgraph} illustrates the bijection from Theorem~\ref{thm:bij_rec_subgraph}. 
The outer cycle on the right is the properly marked orientation $\Phi_W(c)$, as in Figure~\ref{fig:bij_rec_pmor}, and the inner cycle is the corresponding subgraph of the edge-to-vertex dual cycle graph. 
For readability purposes, we have represented vertices present in the subgraph (corresponding to counter-clockwise edges) in black, and ``missing'' vertices (corresponding to clockwise edges) in white.

\begin{figure}[ht]
\centering
\begin{tikzpicture}

\node[rectangle, draw=black] (0) at (0,0) {$*$};
\node[circle, thick, draw=blue] (1) at (0,2) {$1$};
\node[circle, thick, draw=black] (2) at ({sqrt(2)},{sqrt(2)}) {$2$};
\node[circle, thick, draw=blue] (3) at (2,0) {$0$};
\node[circle, thick, draw=blue] (4) at ({sqrt(2)},-{sqrt(2)}) {$1$};
\node[circle, thick, draw=black] (5) at (0,-2) {$2$};
\node[circle, thick, draw=black] (6) at (-{sqrt(2)},-{sqrt(2)}) {$1$};
\node[circle, thick, draw=red] (7) at (-2,0) {$2$};
\node[circle, thick, draw=blue] (8) at (-{sqrt(2)},{sqrt(2)}) {$0$};
\draw[thick] (1)--(2)--(3)--(4)--(5)--(6)--(7)--(8)--(1);
\foreach \x in {1,...,8}
	\draw[thick] (0)--(\x);

\begin{scope}[xshift=6cm]

\node[circle, thick, draw=black] (1) at (0,2) {};
\node[circle, thick, draw=black] (2) at ({sqrt(2)},{sqrt(2)}) {};
\node[circle, thick, draw=black] (3) at (2,0) {};
\node[circle, thick, draw=black] (4) at ({sqrt(2)},-{sqrt(2)}) {};
\node[circle, thick, draw=black] (5) at (0,-2) {};
\node[circle, thick, draw=black] (6) at (-{sqrt(2)},-{sqrt(2)}) {};
\node[circle, thick, draw=red, color=red] (7) at (-2,0) {};
\node[circle, thick, draw=black] (8) at (-{sqrt(2)},{sqrt(2)}) {};
\draw[thick,->,blue] (1)--(2);
\draw[thick,->] (3)--(2);
\draw[thick,->,blue] (3)--(4);
\draw[thick,->,blue] (4)--(5);
\draw[thick,->] (6)--(5);
\draw[thick,->] (7)--(6);
\draw[thick,->] (8)--(7);
\draw[thick,->,blue] (8)--(1);

\node[circle, draw, fill=black] (11) at (-1.25,0.5) {};
\node[circle, draw, fill=black] (12) at (-1.25,-0.5) {};
\node[circle, draw, fill=black] (13) at (-0.55,-1.1) {};
\node[circle, draw=black] (14) at (0.55,-1.1) {};
\node[circle, draw=black] (15) at (1.25,-0.5) {};
\node[circle, draw, fill=black] (16) at (1.25,0.5) {};
\node[circle, draw=black] (17) at (0.55,1.1) {};
\node[circle, draw=black] (18) at (-0.55,1.1) {};

\draw[very thick] (11)--(12);

\end{scope}

\end{tikzpicture}

\caption{Illustrating the bijection from $\SSMRecn$ to $\Subn$.
\label{fig:bij_rec_subgraph}
}

\end{figure}
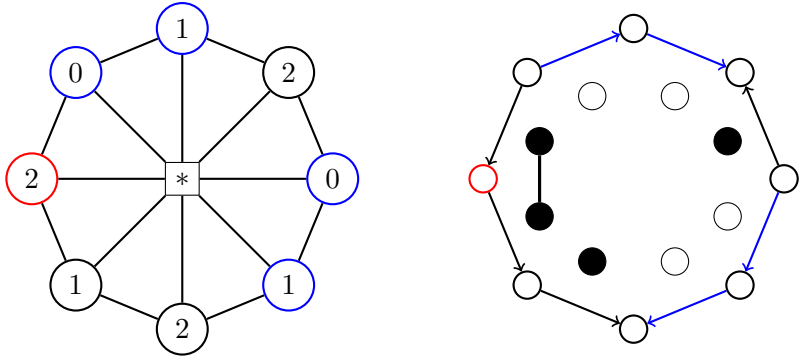

The following straightforward corollary links a specification of the Tutte polynomial of $W_n$ to the number of subgraphs of $C_n$. 
While some quite complicated explicit general formulae have been given for the Tutte polynomial of $W_n$ (see e.g. \cite{BDS, BMM}), and simpler formulae for the chromatic polynomial (another specification of the Tutte polynomial), as far as we know, this particular specification, with its simple combinatorial interpretation, has not been observed previously in the literature.

\begin{corollary}\label{cor: Tutte}
Let $T_{W_n}(x,y)$ denote the Tutte polynomial of the wheel graph $W_n$. We have:
$$ 1 + T_{W_n}(1,x) = \sum\limits_{k \geq 0} a_{k,n} x^k, $$
where $a_{k,n}$ is the number of subgraphs of $C_n$ with $k$ edges.
\end{corollary}

The sequence $\left( a_{k,n} \right)$ is given by Sequence A277919 in the OEIS~\cite{OEIS}, although no explicit formula is given for computing these numbers.

\begin{proof}
This follows immediately from Theorem~\ref{thm:Tutte_level} and Theorem~\ref{thm:bij_rec_subgraph}, along with the observation that $\SSMRecn$ is obtained by adding the configuration $(1,\ldots,1)$ (with level $0$) to $\ASMRecn$ (see Equation~\eqref{eq:ssm_asm_rec_wheel} and preceding remarks).
\end{proof}

\begin{remark}\label{rem:importance_ssm}
While the recurrent configurations for the ASM and SSM on wheel graphs differ only by one element, the configuration $(1,\ldots,1)$, this additional element really is crucial in our construction via canonical minimal recurrent configurations. 
Indeed, any configuration where all vertices have at least one grain maps canonically to this additional element. 
The properly-marked orientation corresponding to such a configuration is simply the counter-clockwise directed cycle, with marks on all vertices with two grains of sand. 
In terms of subgraphs, these correspond to all subgraphs with $n$ vertices, i.e. where all the vertices are present.
\end{remark}


\section{Links to Delannoy numbers}\label{sec:delannoy}

In this section, we consider the case where the cycle has an even number of edges (or vertices) and the subgraphs have exactly half the edges of the cycle. We denote $\Gamma_n$ the set of subgraphs of $C_{2n}$ with $n$ edges. 
Our main goal is to enumerate $\Gamma_n$. In fact, we enumerate these subgraphs according to their number of \emph{connected components}. 
Consider a subgraph $G$ of $C_{2n}$ with $n$ edges. A connected component of $G$ is simply a path with $i$ edges and $i+1$ vertices, for some $i \geq 0$ (isolated vertices are considered to make up a connected component). 
By summing these numbers of vertices and edges over all connected components, we see that if $G$ has $k$ connected components, then it has $(n+k)$ vertices (and vice versa). Note that we must have $1 \leq k \leq n$. 

In the spirit of Figure~\ref{fig:bij_rec_subgraph}, we will refer to the $(n-k)$ vertices not in  $G$ as the ``missing'' (or sometimes ``white'') vertices of $G$. 
For $k \in \{1,\ldots,n\}$, we denote $\Gamma_{n,k}$ the set of subgraphs in $\Gamma_n$ with $k$ connected components (equivalently $(n+k)$ vertices, equivalently $(n-k)$ missing vertices). 
The subgraph on the left of Figure~\ref{fig:balls_boxes} is an element of $\Gamma_{4,3}$.

The sets $\Gamma_{n,k}$ turn out to be equinumerous with certain lattice paths. These paths are named Delannoy paths after Henri Delannoy, a French mathematician of the 19th century. For more information about the life and work of Delannoy, we refer the reader to the very interesting historical note by Banderier and Schwer~\cite{Band}.

A \emph{Delannoy path} is a path in $\Z^2$ starting at $(0,0)$ with steps $R:=(1,0)$, $U:=(0,1)$ and $D:=(1,1)$. A Delannoy path is said to be \emph{symmetric} if it ends on the first diagonal, i.e. at some point $(x,x)$. 
In this paper, we only consider symmetric Delannoy paths, and will typically refer to them as simply Delannoy paths. 
For $n \geq 0$ we denote $\Del[n]$ the set of Delannoy paths ending at $(n,n)$, and for $0 \leq k \leq n$, we denote $\Del[n,k]$ the elements of $\Del[n]$ with $k$ steps $R$ (equivalently $k$ steps $U$, equivalently $(n-k)$ steps $D$). The $n$-th (central) \emph{Delannoy number} is $\vert \Del[n] \vert$.

These numbers appear in several related problems in enumerative combinatorics and also have an interesting link to the Legendre polynomials. 
We refer the interested reader to the fun paper by Sulanke~\cite{Sul} which lists 29 examples of combinatorial objects enumerated by the central Delannoy numbers, such as paths, polyominoes, domino tilings, graph matchings, and so on. 
In this paper, we use the following notion of differed Delannoy paths.

\begin{definition}\label{def:delannoy}
A \emph{differed} Delannoy path is a Delannoy path whose first step is not $D$. We denote $\DiffDel[n]$ the set of differed Delannoy paths ending at $(n,n)$ and $\DiffDel[n,k]$ those with $k$ steps $R$. Note that in this case we should have $k \geq 1$.
\end{definition}

\begin{figure}[ht]
\centering

\begin{tikzpicture}[scale=0.9]
\foreach \x in {0,...,4}
  \foreach \y in {0,...,4}
    \draw [fill=mygray, color=mygray] (\x,\y) circle [radius=0.1];
\draw [dashed, mygray] (-0.3,-0.3)--(4.3,4.3);
\draw (1) [fill] (0,0) circle [radius=0.1];
\draw (2) [fill] (1,0) circle [radius=0.1];
\draw (3) [fill] (1,1) circle [radius=0.1];
\draw (4) [fill] (1,2) circle [radius=0.1];
\draw (5) [fill] (2,3) circle [radius=0.1];
\draw (6) [fill] (3,3) circle [radius=0.1];
\draw (7) [fill] (4,3) circle [radius=0.1];
\draw (8) [fill] (4,4) circle [radius=0.1];
\foreach \x in {1,...,7}
  \draw [thick] (0,0)--(1,0)--(1,2)--(2,3)--(4,3)--(4,4);
\end{tikzpicture}

\caption{A differed Delannoy path in $\DiffDel[4,3]$. \label{fig:delannoy_path_example}}

\end{figure}
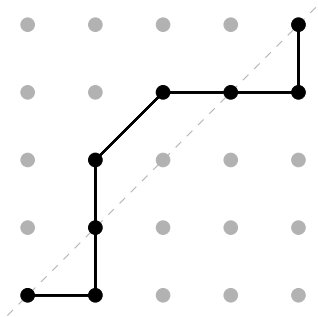

\begin{remark}\label{rem:diff_delannoy}
The term \emph{differed} Delannoy path comes from the following straightforward observation. Partitioning elements of $\Del[n,k]$ according to whether their first step is $D$ or not, we get that $\vert \DiffDel[n,k] \vert = \vert \Del[n,k] \vert - \vert \Del[n-1,k] \vert$ for any $n \geq 1$ and $k \in \{1,\ldots,n\}$. That is, the differed Delannoy paths are counted by the first differences of the Delannoy numbers.
\end{remark}

The general sequence of first differences of Delannoy numbers $\vert \DiffDel[n] \vert$ (not decomposed according to step numbers) is given by Sequence~A110170 in the OEIS~\cite{OEIS}. If we decompose differed Delannoy paths according to their number of steps, we get the following array, with row indices $n \geq 1$ and column indices $k \geq 1$.
$$ \begin{array}{llllll}
2 & & & & & \\
4 & 6 & & & & \\
6 & 24 & 20 & & & \\
8 & 60 & 120 & 70 & & \\
10 & 120 & 420 & 560 & 252 & \\
12 & 210 & 1120 & 2520 & 2520 & 924 \\
\end{array}
$$
We can check that the rows do indeed sum to the first differences of Delannoy numbers given in the OEIS.

\begin{remark}\label{rem:diff_delannoy_oeis}
The sequence of numbers above does not yet appear in the OEIS (the author has submitted it for consideration). The equivalent array for Delannoy numbers rather than first differences is Sequence A063007.
\end{remark}

\begin{theorem}\label{thm:delannoy}
Let $n \geq 1$, and $k \in \{1,\ldots,n\}$. The following three sets all have cardinality 
\beq\label{eq:count_delannoy}
\binom{2k}{k} \cdot \binom{n+k-1}{n-k}.
\eeq
\begin{enumerate}
\item The set of recurrent configurations $c \in \Rec[W_{2n}]$ satisfying $\level(c) = n$ and $\w(c) = n-k$.
\item The set $\Gamma_{n,k}$ of subgraphs of $C_{2n}$ with $n$ edges and $k$ connected components (equivalently $(n+k)$ vertices).
\item The set $\DiffDel[n,k]$ of differed Delannoy paths ending at $(n,n)$ with $k$ steps $R$ (equivalently $(n-k)$ steps $D$).
\end{enumerate}
\end{theorem}

The following lemma is a straightforward classical counting exercise, but as it will be re-used in various parts of the proof of the above theorem, we find it useful to state it here.

\begin{lemma}\label{lem:balls_boxes}
Let $n \geq 1$ and $k \geq 0$. The number of ways of putting $k$ unmarked (indistinguishable) balls into $n$ ordered (distinguishable) boxes is given by the binomial coefficient $\binom{n+k-1}{k}$.
\end{lemma}

\begin{proof}[Sketch of proof]
Including the $(n-1)$ delimiters between the boxes, there are $(n+k-1)$ ``spaces'' (balls + delimiters) in total for us to choose where the $k$ balls go. More formally, if $b = (b_1, \ldots, b_n)$ denotes the balls-in-boxes sequence ($b_i$ being the number of balls in box~$i$), we can define a word $w(b) := 0^{b_1} 1 0^{b_2} 1 \cdots 0^{b_{n-1}} 1 0^{b_n}$, where $0^k = 0\cdots0$ (repeated $k$ times). The map $w$ is a bijection from the desired balls-in-boxes configurations to the set of words on $\{0, 1\}$ with $k$ $0$'s and $(n-1)$ $1$'s.
\end{proof}

\begin{proof}[Proof of Theorem~\ref{thm:delannoy}]
The fact that (1) and (2) are equinumerous is an immediate consequence of Theorem~\ref{thm:bij_rec_subgraph} (a subgraph of $C_{2n}$ with $n+k$ vertices has $(n-k)$ ``missing'' vertices). 
The fact that (2) and (3) are equinumerous when removing the subscript $k$ (keeping only $n$) is noted in the CrossRefs of Entry~A277919 of the OEIS~\cite{OEIS} (``Middle diagonal gives A110170''), but we haven't been able to find a reference to this in the literature, so are unsure if it has been formally proved or not. 
In any case, the specialisation in $k$ does not seem to have been noted.

Our first claim is the following formula enumerating differed Delannoy paths.
\beq\label{eq:count_diff_delannoy}
\big\vert \DiffDel[n,k] \big\vert = \binom{2k}{k} \cdot \binom{n+k-1}{n-k},
\eeq
for all $n \geq 1$ and $ k \in \{1,\ldots,n\}$. Indeed, since a path in $\DiffDel[n,k]$ is comprised of $k$ $R$-steps, $k$ $U$-steps, and $(n-k)$ $D$-steps where the first step is not $D$, we can decompose it as follows. 
First, choose the underlying ``shape'' of the path, that is the path with all $D$ steps removed. There are $\binom{2k}{k}$ such shapes. It remains to insert the $(n-k)$ $D$-steps into a shape. 
Now notice that $D$-steps can be inserted after any $R$- or $U$-step of the shape (but not before the first step), which gives $2k$ places where $D$-steps can be inserted. 
The insertion of $D$-steps is thus equivalent to placing $(n-k)$ unmarked balls into $2k$ ordered boxes, which yields Equation~\eqref{eq:count_diff_delannoy} by Lemma~\ref{lem:balls_boxes}.

We therefore wish to show that $\Gamma_{n,k}$ is also enumerated by the right-hand side of Equation~\eqref{eq:count_diff_delannoy}. 
It turns out to be easier to enumerate a subset of $\Gamma_{n,k}$. Define $\bar{\Gamma}_{n,k}$ to be the subset of graphs $G \in \Gamma_{n,k}$ such that the vertex $1$ is in $G$ but the edge $\{2n,1\}$ is not. 
In words, $\bar{\Gamma}_{n,k}$ is the set of subgraphs of $C_{2n}$ with $n$ edges and $k$ connected components, where the vertex labelled $1$ (the ``top'' vertex in our geometric representation) is in the subgraph, and it is the ``first'' (clockwise) vertex in its connected component. We claim that:
\beq\label{eq:count_graphs}
\big\vert \bar{\Gamma}_{n,k} \big\vert = \binom{n+k-1}{n} \cdot \binom{n-1}{n-k},
\eeq
for all $n \geq 1$ and $ k \in \{1,\ldots,n\}$. To show this, we observe that the construction of a subgraph $G \in \bar{\Gamma}_{n,k}$ is equivalent to placing $n$ unmarked black balls and $(n-k)$ unmarked white balls into $k$ ordered boxes. 
The black balls correspond to the $n$ edges, placed in the $k$ connected components of $G$, while the white balls correspond to the $(n-k)$ vertices in $C_{2n} \setminus G$ (i.e. absent or ``white'' vertices of the subgraph $G$), placed into the $k$ gaps between connected components of $G$. 

We illustrate the construction on an example in Figure~\ref{fig:balls_boxes} below, with $n=4$ and $k=3$. 
Going round the cycle clockwise, the connected component starting at the top vertex has one edge, and there are no missing vertices between this component and the next connected component, so we put one black ball (for the edge) and no white balls (for the missing vertices) in Box 1. 
Then the second connected component has three edges, and is followed by one missing (white) vertex, so we put three black balls and one white ball in Box 2. 
Finally, the third connected is an isolated vertex and is followed by no missing vertices, so Box 3 is empty. This correspondence is clearly one-to-one (it is straightforward to construct its inverse), and Equation~\eqref{eq:count_graphs} follows by once again applying Lemma~\ref{lem:balls_boxes}.

\begin{figure}[ht]
\centering

\begin{tikzpicture}

\node[circle, thick, draw=black, fill=black] (1) at (0,2) {};
\node[circle, thick, draw=black, fill=black] (2) at ({sqrt(2)},{sqrt(2)}) {};
\node[circle, thick, draw=black, fill=black] (3) at (2,0) {};
\node[circle, thick, draw=black, fill=black] (4) at ({sqrt(2)},-{sqrt(2)}) {};
\node[circle, thick, draw=black, fill=black] (5) at (0,-2) {};
\node[circle, thick, draw=black, fill=black] (6) at (-{sqrt(2)},-{sqrt(2)}) {};
\node[circle, thick, draw=black] (7) at (-2,0) {};
\node[circle, thick, draw=black, fill=black] (8) at (-{sqrt(2)},{sqrt(2)}) {};
\draw[thick] (1)--(2);
\draw[thick] (3)--(4)--(5)--(6);

\begin{scope}[xshift=-1.3cm]
\draw[thick] (5,1)--(5,-1)--(6.5,-1)--(6.5,1);
\node[above, yshift=6, circle, draw, fill=black] (b1) at (5.75, -1) {};
\node[below, yshift=-6] at (5.75,-1) {Box 1};

\draw[thick] (7.5,1)--(7.5,-1)--(9,-1)--(9,1);
\node[above, yshift=6, circle, draw, fill=black] (b2) at (8, -1) {};
\node[above, yshift=6, circle, draw, fill=black] (b3) at (8.5, -1) {};
\node[yshift=-2, circle, draw, fill=black] (b4) at (8, 0) {};
\node[yshift=-2, circle, draw=black] (w1) at (8.5, 0) {};
\node[below, yshift=-6] at (8.25,-1) {Box 2};

\draw[thick] (10,1)--(10,-1)--(11.5,-1)--(11.5,1);
\node[below, yshift=-6] at (10.75,-1) {Box 3};
\end{scope}

\end{tikzpicture}

\caption{Illustrating the correspondence between $\bar{\Gamma}_{n,k}$ and placing $n$ unmarked black balls and $(n-k)$ unmarked white balls into $k$ ordered boxes, for $n=4$ and $k=3$.\label{fig:balls_boxes}}

\end{figure}
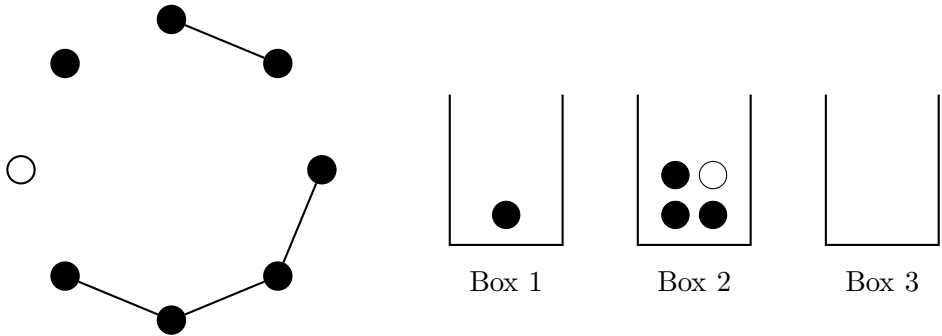

Finally, we show the following formula linking the enumeration of this subset $\bar{\Gamma}_{n,k}$ to the set $\Gamma_{n,k}$. This formula can be thought of as a type of cyclic lemma for graphs.
\beq\label{eq:rooted_graphs}
(2n) \cdot \big\vert \bar{\Gamma}_{n,k} \big\vert = k \cdot \big\vert \Gamma_{n,k} \big\vert,
\eeq
for all $n \geq 1$ and $ k \in \{1,\ldots,n\}$.
The way this is written is suggestive. 
The left-hand side counts the set of pairs $(G,i) \in \bar{\Gamma}_{n,k} \times [2n]$, i.e. the set of graphs $G \in \bar{\Gamma}_{n,k}$ rooted at a vertex of $C_{2n}$. 
Similarly, the right-hand side counts the set of graphs $G \in \Gamma_{n,k}$ rooted at a connected component of $G$. 

In fact, it is slightly more convenient to think of the two sides counting certain \emph{doubly-rooted} subgraphs of $C_{2n}$. A doubly-rooted subgraph is a triple $\mathcal{G} = (G,i,c)$ where $G \in \Gamma_{n,k}$ is a sub-graph of $C_{2n}$, $i \in [2n]$ is a vertex of $C_{2n}$ (not necessarily a vertex of the subgraph $G$), and $c$ is a connected component of $G$. Now the left-hand side of Equation~\eqref{eq:rooted_graphs} counts doubly-rooted subgraphs $(G,i,c)$ such that $G \in \bar{\Gamma}_{n,k}$ (i.e. $1$ is the clockwise first vertex of the connected component $c$). An example is shown on the left of Figure~\ref{fig:mark_graphs}. The right-hand side of Equation~\eqref{eq:rooted_graphs} counts doubly-rooted subgraphs $(G,i,c)$ such that $i=1$ (i.e. the ``top'' vertex is the root). We exhibit a simple bijection between these two sets of doubly-rooted subgraphs.

Consider a doubly-rooted subgraph $\mathcal{G} = (G,i,c)$ with $G \in \bar{\Gamma}_{n,k}$, i.e. $1$ is the clockwise first vertex of the connected component $c$.
We construct a doubly-rooted subgraph $
\mathcal{G}' = (G',i',c')$ by rotating $G$ anti-clockwise through $i-1$ vertices, so that the vertex that was previously labelled $i$ is now labelled $i'=1$. The connected component $c'$ is the image of $c$ under this rotation, i.e. a path of same length as $c$, whose clockwise first vertex is now $(2n+1-i)$.

The map $\mathcal{G} \mapsto \mathcal{G'}$ gives the desired bijection. By construction it maps a doubly-rooted subgraph whose connected component root starts at vertex $1$ to a doubly-rooted subgraph whose root vertex is $1$. Furthermore, the inverse bijection is simply the inverse rotation. Namely, starting from $\mathcal{G}' = (G',1,c')$, we get the inverse image $\mathcal{G} = (G, i, c)$ by rotating $\mathcal{G}'$ so that its connected component root starts at vertex $1$. 

The bijection between rooted graphs is illustrated on Figure~\ref{fig:mark_graphs} below. 
On the left, a doubly-roooted subgraph $\mathcal{G} = (G,i,c)$ with $G \in \bar{\Gamma}_{n,k}$. We represent the vertex root in blue (here labelled $i = 7$), and the connected component root (which here starts at the top vertex $j=1$) in red.
On the right, the corresponding doubly-rooted subgraph $
\mathcal{G}' = (G',1,c')$, obtained by rotating $\mathcal{G}$ counterclockwise through $i-1 = 6$ vertices (or equivalently clockwise through $2$ vertices), so that the vertex root is now the top vertex labelled $i'=1$, and the connected component root starts at vertex labelled $j' = 2n + 1 - i = 3$.

\begin{figure}[ht]
\centering

\begin{tikzpicture}

\begin{scope}[xshift=-6cm]
\node[circle, thick, draw=red, fill=red] (1) at (0,2) {};
\node[circle, thick, draw=red, fill=red] (2) at ({sqrt(2)},{sqrt(2)}) {};
\node[circle, thick, draw=red, fill=red] (3) at (2,0) {};
\node[circle, thick, draw=black] (4) at ({sqrt(2)},-{sqrt(2)}) {};
\node[circle, thick, draw=black, fill=black] (5) at (0,-2) {};
\node[circle, thick, draw=black, fill=black] (6) at (-{sqrt(2)},-{sqrt(2)}) {};
\node[circle, thick, draw=blue] (7) at (-2,0) {};
\node[circle, thick, draw=black, fill=black] (8) at (-{sqrt(2)},{sqrt(2)}) {};
\draw[thick] (6)--(5);
\draw[thick,red] (1)--(2)--(3);
\node[left, xshift=-5pt] at (-2,0) {\textcolor{blue}{$i=7$}};
\node[above, yshift=4pt] at (0,2) {\textcolor{red}{$j=1$}};
\end{scope}

\node[circle, thick, draw=blue] (1) at (0,2) {};
\node[circle, thick, draw=black, fill=black] (2) at ({sqrt(2)},{sqrt(2)}) {};
\node[circle, thick, draw=red, fill=red] (3) at (2,0) {};
\node[circle, thick, draw=red, fill=red] (4) at ({sqrt(2)},-{sqrt(2)}) {};
\node[circle, thick, draw=red, fill=red] (5) at (0,-2) {};
\node[circle, thick, draw=black] (6) at (-{sqrt(2)},-{sqrt(2)}) {};
\node[circle, thick, draw=black, fill=black] (7) at (-2,0) {};
\node[circle, thick, draw=black, fill=black] (8) at (-{sqrt(2)},{sqrt(2)}) {};
\node[above, yshift=4pt] at (0,2) {\textcolor{blue}{$i'=1$}};
\node[right, xshift=5pt] at (2,0) {\textcolor{red}{$j'=3$}};
\draw[thick] (7)--(8);
\draw[thick,red] (3)--(4)--(5);

\end{tikzpicture}

\caption{Illustrating the bijection that established Equation~\eqref{eq:rooted_graphs} \label{fig:mark_graphs}.}

\end{figure}
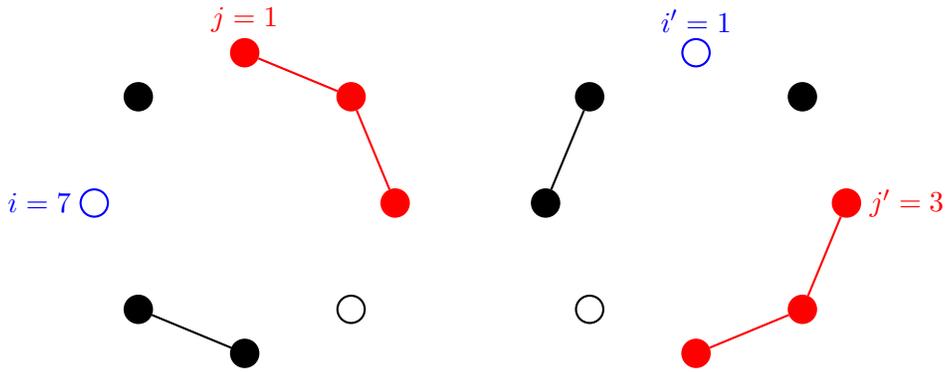

Combining Equations~\eqref{eq:count_diff_delannoy}, \eqref{eq:count_graphs} and \eqref{eq:rooted_graphs}, it suffices to check the following binomial equality
\beq\label{eq:binom_eq}
 \dfrac{2n}{k} \cdot \binom{n+k-1}{n} \cdot \binom{n-1}{n-k} = \binom{2k}{k} \cdot \binom{n+k-1}{n-k},
\eeq
which can be done by simply expanding the binomial coefficients. This completes the proof of Theorem~\ref{thm:delannoy}.
\end{proof}

\begin{remark}\label{rem:bij_proof_rec_del}
While our proofs of the combinatorial equalities~\eqref{eq:count_diff_delannoy}, \eqref{eq:count_graphs} and \eqref{eq:rooted_graphs} are all bijective, and a bijective proof can undoubtedly be found of the binomial equality~\eqref{eq:binom_eq}, we have been unable to find a direct bijection between the sets $\Gamma_{n,k}$ and $\DiffDel[n,k]$. 
It would be useful if such a bijection (offering a direct bijective proof of Theorem~\ref{thm:delannoy}) could be exhibited in the future.
\end{remark}


\section{The sandpile model on fan graphs}\label{sec:sandpile_fan}

In this section, we study the sandpile model on fan graphs $F_n$. 
Our main goals in this section are to show that the recurrent configurations on $F_n$ are in bijection with certain subgraphs of the path graph $P_n$, and that they are counted by certain lattice paths. As in the wheel graph case, this specialises according to the level of the configuration. 
As in Sections~\ref{sec:sandpile_wheel} and \ref{sec:delannoy}, it is convenient to consider the geometric representations of $F_n$ and $P_n$ where $P_n$ is simply a line of vertices $1,\ldots,n$ from left-to-right, and thus $F_n$ is as in Figure~\ref{fig:cycle_wheel_fan}. 
The tools we use are similar to those of these two sections, so we will allow ourselves to be briefer here. We proceed as follows.

Firstly, in Section~\ref{subsec:rec_fan_char_enum}, we define a notion of canonical minimal recurrent configuration, as in Section~\ref{sec:sandpile_wheel}. This allows us to exhibit a bijection between $\Rec[F_n]$ and a set of properly-marked words on the alphabet $\{L,R\}$. We then exhibit a natural bijection between these properly-marked words and subgraphs of $P_n$ whose vertex set contains the right-most vertex $n$.
In Section~\ref{subsec:kimb}, we partition these properly-marked words according to their number of unmarked $L$-elements. Using similar decompositions to those in the proof of Theorem~\ref{thm:delannoy}, we then show that they are enumerated by the lattice paths in question.

\begin{remark}\label{rem:SR_DR_fan}
From Theorem~\ref{thm:MR_Or}, Remark~\ref{rem:sink-rooted_or_wheel_fan} and the observation that the path graph $P_n$ is acyclic (hence so are all its orientations), we deduce that the sets of recurrent configurations for the ASM and the SSM on $F_n$ are the same. 
Thus, we will simply talk of recurrent configurations, and denote $\Rec[F_n]$ their set.
\end{remark}

\subsection{Characterisation and enumeration of \texorpdfstring{$\Rec[F_n]$}{Rec(Fn)}}\label{subsec:rec_fan_char_enum}

In the spirit of Theorem~\ref{thm:wheel_rec_words}, we have a similar characterisation of the set $\Rec[F_n]$. 
Note that this time $(i,j)$ refers to the standard open integer interval, rather than the cyclic interval of the wheel graph case.

\begin{theorem}\label{thm:fan_rec_words}
Let $c \in \{0,1,2\}^n$ with $c_1,c_n \leq 1$ be a stable configuration on $F_n$. Then $c$ is recurrent if, and only if, for all pairs $i,j \in [n]$ such that $i<j$ and $c_i=c_j=0$, there exists $k \in (i,j)$ such that $c_k = 2$. 
\end{theorem}

The proof of this result is essentially analogous to that of Theorem~\ref{thm:wheel_rec_words}, with minor adjustments to take into account the removal of the cycle edge $\{n,1\}$, so we omit it here. 
Similarly, the minimal recurrent configurations are configurations with at least one pair $(i,j)$ such that $c_i=c_j=0$ and for each such consecutive pair the $k$ is unique, as well as configurations where exactly one vertex has no grains, and all other vertices have (exactly) one grain.

As in Section~\ref{sec:sandpile_wheel}, we construct a bijection between recurrent configurations on $F_n$ and certain marked orientations of the path $P_n$. We proceed as follows. Consider the path graph $P_n^0$ with an extra vertex $0$ and edge $\{0,1\}$. Given a recurrent configuration $c \in \Rec[F_n]$, we define a marked orientation $(\OO,M)$ of $P_n^0$ as follows:
\begin{enumerate}
\item Direct the edge $0 \olefta 1$.
\item For $i \geq 1$, assume we have given a direction to $\{i-1,i\}$ in $\OO$.
  \begin{enumerate}
  \item If $i-1 \olefta i$, then direct $i \orighta i+1$ if $c_i=0$. Otherwise, direct $i \olefta i+1$, and mark vertex $i$ if $c_i = 2$.
  \item\label{case} If $i-1 \orighta i$, then direct $i \orighta i+1$ if $c_i=1$ and $i \olefta i+1$ if $c_i=2$.
  \end{enumerate}
\item Mark vertex $n$ if $n-1 \olefta n$ and $c_n=1$.
\end{enumerate}

Note that Theorem~\ref{thm:fan_rec_words} implies that $c_i=0$ is impossible in Case~\ref{case}. 
In words, the marked orientation is constructed left-to-right as follows. So long as $c_i \geq 1$, direct $i \olefta i+1$ i.e. ``leftwards'', marking vertex $i$ if $c_i=2$. 
When encountering $i$ such that $c_i=0$, start directing all subsequent edges $j \orighta j+1$ i.e. ``rightwards'' for $j \geq i$ until encountering the first $j$ such that $c_j = 2$ (or reaching the end vertex $n$), at which point we start directing edges leftwards again, and repeat the process. 
Mark the rightmost vertex $n$ if the last edge $\{n-1,n\}$ is directed leftwards and $c_n = 1$.

Removing the (always leftwards) edge $\{0,1\}$, we obtain a marked orientation $(\OO,M)$ of the path graph $P_n$ where all marked vertices $i$ must have $i+1 \orighta i \orighta i-1$ (with the convention that $n+1 \orighta n$). 
As in the wheel graph case, the underlying orientation $\OO$ defines a minimal recurrent configuration on $F_n$, which can be viewed as the canonical minimal recurrent configuration associated with $c$, while the marked vertices represent the locations of ``excess grains'' in $c$ (see Remark~\ref{rem:excess_grains}).

In the path graph case, it is slightly more convenient to mark the edges instead of the vertices. 
For this, if a vertex $i$ is marked, we simply move the mark to its left edge $\{i-1,i\}$ (necessarily directed $i-1 \olefta i$). Since the leftmost vertex $1$ can never be marked by construction ($c_1 < 2$ since $c$ is stable), there is no issue in doing this. 
Transforming the marked orientation into a word on the alphabet $\{L,R\}$ with markings on certain $L$-elements yields the following.

\begin{definition}\label{def:pm_words}
A \emph{properly-marked $\{L,R\}$-word} is a word on the alphabet $\{L^u, L^m, R\}$ such that no occurrence of $L^m$ is followed by an occurrence of $R$. 
We refer to the occurrences of $L^m$ as the \emph{marked} $L$-elements (or simply marked elements), and the occurrences of $L^u$ as the \emph{unmarked} $L$-elements (or simply unmarked elements). 
The set of properly marked $\{L,R\}$-words of length $n$ is denoted $\PMW_n$.
\end{definition}

Note that the condition is equivalent to every marked $L$-element being either followed by another $L$-element (marked or not), or being the last element in the word. We have therefore constructed a map $\Phi_F:\Rec[F_n] \rightarrow \PMW_{n-1}$.  The equivalent of Theorem~\ref{thm:bij_rec_marked_or} is the following result (again, the proof is analogous, so we omit it here).

\begin{theorem}\label{thm:bij_rec_fan_pm_words}
The map $\Phi_F:\Rec[F_n] \rightarrow \PMW_{n-1}$ is a bijection from the set of recurrent configurations on the fan graph $F_n$ to the set of properly-marked $\{L,R\}$-words of length $(n-1)$. Moreover, for $c \in \Rec[F_n]$, the level of the configuration $c$ is equal to the number of marked elements in the word $\Psi_F(c)$.
\end{theorem}

An example of the bijection $\Phi_F$, applied to the recurrent configuration (read from left to right) $c = (1,1,0,1,2,2,0,2,1) \in \Rec[F_9]$, is illustrated in Figure~\ref{fig:bij_rec_fan_marked_word}. 
The marked orientation of the path $P_9$ is shown, with marks (in red) on both the vertices and their immediate left edges for convenience. 
This maps to the properly-marked $\{L,R\}$-word $w = \Phi_F(c) \in \PMW_8$, which is shown above the orientation ($w = L^uL^uRRL^mL^uRL^m$).

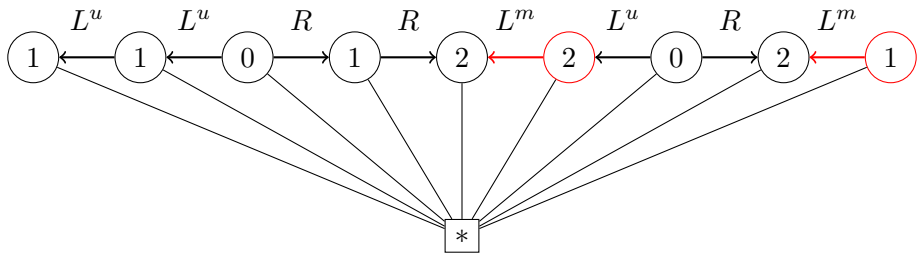
\begin{figure}[ht]

\centering

\begin{tikzpicture}[scale=0.95]

\node[rectangle, draw=black] (0) at (0,-2.5) {$*$};
\node[circle, draw=black] (1) at (-6,0) {$1$};
\node[circle, draw=black] (2) at (-4.5,0) {$1$};
\node[circle, draw=black] (3) at (-3,0) {$0$};
\node[circle, draw=black] (4) at (-1.5,0) {$1$};
\node[circle, draw=black] (5) at (-0,0) {$2$};
\node[circle, draw=red] (6) at (1.5,0) {$2$};
\node[circle, draw=black] (7) at (3,0) {$0$};
\node[circle, draw=black] (8) at (4.5,0) {$2$};
\node[circle, draw=red] (9) at (6,0) {$1$};
\foreach \x in {1,...,9}
	\draw (0)--(\x);
\draw[thick,->] (2)--(1);
\draw[thick,->] (3)--(2);
\draw[thick,->] (3)--(4);
\draw[thick,->] (4)--(5);
\draw[thick,->,red] (6)--(5);
\draw[thick,->] (7)--(6);
\draw[thick,->] (7)--(8);
\draw[thick,->,red] (9)--(8);

\node[yshift=0.5cm] (w1) at (-5.25,0) {$L^u$};
\node[yshift=0.5cm] (w2) at (-3.75,0) {$L^u$};
\node[yshift=0.5cm] (w3) at (-2.25,0) {$R$};
\node[yshift=0.5cm] (w4) at (-0.75,0) {$R$};
\node[yshift=0.5cm] (w5) at (0.75,0) {$L^m$};
\node[yshift=0.5cm] (w6) at (2.25,0) {$L^u$};
\node[yshift=0.5cm] (w7) at (3.75,0) {$R$};
\node[yshift=0.5cm] (w8) at (5.25,0) {$L^m$};

\end{tikzpicture}

\caption{Illustrating the bijection $\Phi_F$.
\label{fig:bij_rec_fan_marked_word}
}

\end{figure}

As in the wheel graph case, there is a subgraph interpretation of properly-marked words. Let $w = w_1 \cdots w_{n-1}$ be a properly-marked $\{L,R\}$-word of length $(n-1)$. We construct a subgraph $G = (V,E) = T(w)$ of the path graph $P_n$ by letting $V = \big\lbrace i \in [n]; \, w_i \in \{L^u, L^m\} \big\rbrace \cup \{ n \}$, and $E = \big\lbrace \{i, i+1\}; w_i = L^m \big\rbrace$.For example, if $w = L^uL^uRRL^mL^uRL^m$ is the properly-marked word from Figure~\ref{fig:bij_rec_fan_marked_word}, then the corresponding subgraph has vertex set $V = \{1, 2, 5, 6, 8, 9 \}$ and edge set $E = \big\lbrace \{5, 6\}, \{8, 9\} \big\rbrace$.

In words, if we think of the word $w$ as labelling the edges of $P_n$ as in Figure~\ref{fig:bij_rec_fan_marked_word}, then the vertices of $T(w)$ are those whose right-incident edge is labelled with an $L$-element (marked or not), as well as the right-most vertex $n$, while the edges are simply the edges labelled $L^m$. Since each $L^m$ is either followed by an $L$-element, or is the right-most edge, this implies that the map is well defined, i.e. that $T(w)$ is indeed a subgraph of $P_n$.

Moreover, given a subgraph $G = (V,E)$ of $P_n$ such that $n \in V$, we can define its inverse $w = w_1 \cdots w_{n-1} = T'(G)$ by setting:
$$ w_i = \begin{cases} 
L^m & \text{if } \{i, i+1\} \in E \\
L^u & \text{if } i \in V \mbox{ and } \{i, i+1\} \notin E \\
R & \text{otherwise (i.e. if } i \notin V \mbox{)}.
\end{cases} $$
By construction, $T'(G)$ is a properly-marked $\{L,R\}$-word, and the maps $T$ and $T'$ are inverses of each other. We also note that for $i < n$, then $i$ is the right-most vertex of a connected component of $G$ if, and only if, $w_i = L^m$. This implies that the number of connected components of $G$ is equal to the number of unmarked $L$-elements plus one (for the right-most connected component, whose right-most vertex is always $n$). We summarise this in the following proposition.

\begin{proposition}\label{pro:bij_pmword_subgraph}
The map $T: w \mapsto T(w)$ is a bijection between the set of properly-marked $\{L,R\}$-word of length $(n-1)$ and the set of subgraphs of the path graph $P_n$ whose vertex set contains $n$. Moreover, the number of marked, resp.\ unmarked, $L$-elements of $w$ is equal to the number of edges, resp.\ to one plus the number of connected components, of $T(w)$.
\end{proposition}

Proposition~\ref{pro:bij_pmword_subgraph} and Theorem~\ref{thm:bij_rec_fan_pm_words} immediately combine for the following.

\begin{theorem}\label{thm:rec_fan_subgraph_path}
The map $T \circ \Phi_F$ defines a bijection from the set of recurrent configurations on the fan graph $F_n$ to the set of subgraphs of the path graph $P_n$ whose vertex set contains $n$. Moreover, this bijection maps the level of a recurrent configuration $c$ to the number of edges in the corresponding subgraph $G$.
\end{theorem}

We now provide a counting formula for the set of properly-marked words $\PMW_n$. In fact, we count these according to the numbers of both marked and unmarked $L$-elements. This also counts subgraphs of $P_n$ containing $n$ according to their numbers of edges and connected components, by Proposition~\ref{pro:bij_pmword_subgraph}.

\begin{proposition}\label{pro:pm_words_enumeration}
Let $n \geq 1$ and $k,r \geq 0$ such that $k+r \leq n-1$. The number of properly-marked $\{L,R\}$-words of length $(n-1)$ with $k$ marked $L$-elements and $r$ unmarked $L$-elements is equal to $\binom{n-k-1}{r} \cdot \binom{r+k}{r}$.
\end{proposition}

\begin{proof}
The proof is similar to that of the enumeration of differed Delannoy paths in Equation~\eqref{eq:count_diff_delannoy}. 
First we choose the underlying ``shape'' of the word, that is the word with marked elements removed. 
There are $\binom{n-1-k}{r}$ possibilities for this (removing the $k$ marked $L$-elements from the word leaves $(n-1-k)$ elements to place, of which $r$ are unmarked $L$-elements and the remainder $R$-elements).

We then need to insert the $k$ marked elements into this shape. But there are $(r + 1)$ places where we can insert such elements: to the left of any of the $r$ unmarked element, and in the rightmost position(s) of the word. 
Thus the possible insertions into a given ``shape'' are equivalent to placing $k$ unmarked balls into $(r+1)$ ordered boxes, and the result follows by Lemma~\ref{lem:balls_boxes}.
\end{proof}

By combining Proposition~\ref{pro:pm_words_enumeration} and Theorem~\ref{thm:bij_rec_fan_pm_words}, we immediately get the following enumeration formula for $\Rec[F_n]$.

\begin{corollary}\label{cor:fan_rec_enumeration}
Let $n \geq 1$ and $0 \leq k < n$. Then:
$$ \big\vert \{ c \in \Rec[F_n]; \, \level(c) = k \} \big\vert = \sum\limits_{r=0}^{n-k-1} \binom{n-k-1}{r} \cdot \binom{r+k}{r}. $$
\end{corollary}

\subsection{Kimberling paths}\label{subsec:kimb}

For $n \geq 1$ and $0 \leq k \leq (n-1)$, the numbers $R(n,k) := \big\vert \{ c \in \Rec[F_n]; \, \level(c) = k \} \big\vert$ of recurrent configurations for the fan graph $F_n$ with level $k$ give the following enumeration triangle.

$$ \begin{array}{lllllll}
1 & & & & & \\
1 & 2 & & & & & \\
1 & 3 & 4 & & & & \\
1 & 4 & 8 & 8 & & & \\
1 & 5 & 13 & 20 & 16 & & \\
1 & 6 & 19 & 38 & 48 & 32 & \\
1 & 7 & 26 & 63 & 104 & 112 & 64 \\
\end{array}
$$

This is (essentially) Sequence A049600 in the OEIS~\cite{OEIS}, and enumerates certain lattice paths. 
Hetyei~\cite{Het} calls a variant of this array the \emph{asymmetric Delannoy} numbers and shows that they are related to a face-enumeration problem associated with Jacobi polynomials. 
Callan~\cite{Cal} calls the lattice paths enumerated by this sequence \emph{Kimberling} paths, after the author of the aforementioned sequence in the OEIS, and we follow this latter terminology here.

\begin{definition}\label{def:kimb}
A \emph{Kimberling path} is a lattice path starting at the origin $(0,0)$ and with steps in $\N \times \Zp$. In other words, Kimberling paths are made of steps with finite non-negative slope (so horizontal rightwards steps are allowed, but vertical upward steps are not). Note that steps in a Kimberling path can have any length. For $i>0, j \geq 0$, we denote $\Kimb[i,j]$ the set of Kimberling paths ending at $(i,j)$. If $0 \leq r \leq i-1$, we denote $\Kimb[i,j,r] \subseteq \Kimb[i,j]$ the set of such paths with $r$ internal vertices (excluding start and end points), equivalently $(r+1)$ steps.
\end{definition}

\begin{figure}[ht]
\centering

\begin{tikzpicture}

\latticegrid
\draw [thick] (0,0)--(2,1);

\begin{scope}[xshift=4.5cm]
\latticegrid
\draw [fill] (1,0) circle [radius=0.1];
\draw (0,0)--(1,0)--(2,1);
\end{scope}

\begin{scope}[xshift=9cm]
\latticegrid
\draw [fill] (1,1) circle [radius=0.1];
\draw (0,0)--(1,1)--(2,1);
\end{scope}

\end{tikzpicture}

\caption{The three paths in $\Kimb[2,1]$. There are three paths in total, one with no internal vertices (left) and two with one internal vertex (middle and right).}

\end{figure}
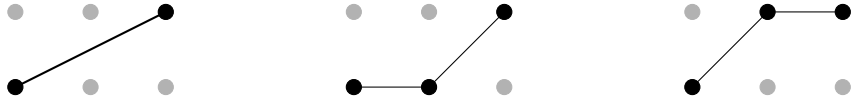

The following gives a simple counting formula for Kimberling paths (see e.g. \cite[Lemma 3.2]{Het}.

\begin{proposition}\label{pro:kimb_enumeration}
Let $i \geq 1, j \geq 0$ and $r \in \{0,\ldots,i-1\}$. Then:
$$ \big\vert \Kimb[i,j,r] \big\vert = \binom{i-1}{r} \cdot \binom{r+j}{r}.$$
\end{proposition}

Combined with Proposition~\ref{pro:pm_words_enumeration}, we immediately get the following.

\begin{corollary}\label{cor:pm_words_kimb}
Let $n \geq 1$ and $k,r \geq 0$ such that $k+r \leq n-1$. The following sets both have cardinality $\binom{n-k-1}{r} \cdot \binom{r+k}{r}$.
\begin{enumerate}
\item The set of properly-marked $\{L,R\}$-words of length $(n-1)$ with $k$ marked elements and $r$ unmarked elements.
\item The set of Kimberling paths ending at $(n-k,k)$ with $r$ internal vertices.
\end{enumerate}
\end{corollary}

Combining with the results of Section~\ref{subsec:rec_fan_char_enum} gives us the following result, which was the main objective of this section.

\begin{theorem}\label{thm:rec_fan_kimb}
Let $n \geq 1$ and $0 \leq k \leq n-1$. The following sets are equinumerous.
\begin{enumerate}
\item The set of recurrent configurations $c \in \Rec[F_{n}]$ satisfying $\level(c) = k$.
\item The set of subgraphs $G = (V,E)$ of the path graph $P_n$ such that $n \in V$ and $\vert E \vert = k$.
\item The set $\Kimb[n-k,k]$ of Kimberling paths ending at $(n-k,k)$.
\end{enumerate}
In particular, if $T_{F_n}$ is the Tutte polynomial of the fan graph $F_n$, we have:
$$ T_{F_n}(1,x) = \sum\limits_{k=0}^{n-1} \big\vert \Kimb[n-k,k] \big\vert \cdot x^k. $$
\end{theorem}

As far as we can tell, as in the wheel graph case, this specification of the Tutte polynomial for fan graphs has not been noted anywhere in the literature. Note that, unlike in the wheel graph case, we can give an explicit formula for the coefficients here:
$$ \big\vert \Kimb[n-k,k] \big\vert = \sum\limits_{r=0}^{n-k-1} \binom{n-k-1}{r} \cdot \binom{r+k}{r}. $$

\begin{remark}\label{rem:interpret_rec_fan_unmarked_elements}
As in the wheel graph case, it is possible to give an interpretation of the parameter $r$ in terms of certain $2\{1,2\}^*0$-chains of the corresponding recurrent configuration, with suitable precautions taken to deal with the end vertices $1$ and $n$.
\end{remark}

\begin{remark}\label{rem:bij_proof_fan}
As in Section~\ref{sec:delannoy}, our proof of Theorem~\ref{thm:rec_fan_kimb} is not a direct bijection. Instead it relies on the fact that both objects are enumerated by the same formula (a product of binomial coefficients in this case). It would be useful to exhibit a direct bijection from either recurrent configurations on $F_n$ or properly-marked $\{L,R\}$-words to the Kimberling paths.
\end{remark}


\section{Conclusion and future work}\label{sec:conc}

In this paper, we have studied the sandpile model on wheel and fan graphs, focusing on bijective characterisations of the recurrent configurations. These results are in the same spirit as a growing list of similar combinatorial characterisations on other graph families (see the list in Section~\ref{sec:intro} for more details). 

In the case of the wheel graph $W_n$, it proved helpful to consider a stochastic variant of the standard Abelian sandpile model (ASM). We described (Theorem~\ref{thm:bij_rec_subgraph}) a bijection between the recurrent configurations and the set of non-empty subgraphs of the cycle graph $C_n$. This bijection maps the level of a configuration to the number of edges of the subgraph, and a statistic called the $01*$-weight of the configuration to the number of vertices absent from the subgraph. The tools used included a bijection between \emph{minimal} recurrent configurations on the wheel graph and orientations of the cycle, which we enriched to a bijection between all recurrent configurations and \emph{properly-marked} orientations of the cycle.

We then considered the special case where the wheel graph has an even number of vertices $2n$. We showed (Theorem~\ref{thm:delannoy}) that the recurrent configurations with level $n$ are equinumerous with certain lattice paths called \emph{differed Delannoy paths}. In this case, the $01*$-weight corresponds to the number of diagonal steps in the path. However, our proof of this result is non-bijective, relying instead on a variety of enumerative arguments. In the future, it would be useful to find a direct bijective proof of this result.

In the case of the fan graph $F_n$, we described (Theorem~\ref{thm:rec_fan_subgraph_path}) a bijection between the recurrent configurations and subgraphs of the path graph $P_n$ whose vertex set contains the right-most vertex of the path. The bijection again maps the level of the recurrent configuration to the number of edges in the corresponding subgraph. The initial tools used are similar to the wheel graph case, relying on enriching a bijection between minimal recurrent configurations on the fan graph and orientations of the path graph, which yields a bijection between all the recurrent configurations and objects called \emph{properly-marked words}. The bijection to the subgraphs of $P_n$ follows naturally. We then show that these sets are equinumerous with another family of lattice paths, which we called \emph{Kimberling paths}, named after the author of the corresponding entry in the OEIS~\cite{OEIS}. As in the wheel graph case, this proof is non-bijective, and in the future it would be useful to find a direct bijection instead.

Because of the well-known link between the enumeration of recurrent configurations for the ASM according to their level, and the ubiquitous polynomial (Theorem~\ref{thm:Tutte_level}), these results yield specialisations of the Tutte polynomial of wheel and fan graphs. It would be interesting to investigate whether the bijective correspondences exhibited here can be enriched in some way to account for the second variable of the Tutte polynomial. This would potentially yield simpler and more explicit formulae for the Tutte polynomial of wheel and fan graphs than those that are currently known.

\section*{Acknowledgments}

The research leading to these results received funding from the National Natural Science Foundation of China (NSFC) under Grant Agreement No 12101505. In addition, we would like to thank the two anonymous referees for their many helpful comments and corrections, as well as the suggestion to look into subgraph characterisations of the recurrent configurations in the fan graph case.


\bibliographystyle{plain}
\bibliography{sandpile_wheel_bibliography}

\begin{thebibliography}{10}

\bibitem{AD}
Amal Alofi and Mark Dukes.
\newblock Parallelogram polyominoes and rectangular ew-tableaux:
  Correspondences through the sandpile model.
\newblock {\em Enumer. Combin. Appl.}, 1(1):Art--S2R8, 2018.

\bibitem{AADB}
Jean-Christophe {Aval}, Michele {D'Adderio}, Mark {Dukes}, and Yvan {Le
  Borgne}.
\newblock {Two operators on sandpile configurations, the sandpile model on the
  complete bipartite graph, and a cyclic lemma}.
\newblock {\em {Adv. Appl. Math.}}, 73:59--98, 2016.

\bibitem{BTW1}
Per {Bak}, Chao {Tang}, and Kurt {Wiesenfeld}.
\newblock {Self-organized criticality: An explanation of the 1/f noise}.
\newblock {\em {Phys. Rev. Lett.}}, 59:381--384, 1987.

\bibitem{BTW2}
Per {Bak}, Chao {Tang}, and Kurt {Wiesenfeld}.
\newblock {Self-organized criticality}.
\newblock {\em {Phys. Rev. A (3)}}, 38(1):364--374, 1988.

\bibitem{Baker}
Matthew {Baker} and Serguei {Norine}.
\newblock {Riemann-Roch and Abel-Jacobi theory on a finite graph}.
\newblock {\em {Adv. Math.}}, 215(2):766--788, 2007.

\bibitem{Band}
Cyril {Banderier} and Sylviane {Schwer}.
\newblock {Why Delannoy numbers?}
\newblock {\em {J. Stat. Plann. Inference}}, 135(1):40--54, 2005.

\bibitem{Ber}
Olivier {Bernardi}.
\newblock {Tutte polynomial, subgraphs, orientations and sandpile model: new
  connections via embeddings}.
\newblock {\em {Electron. J. Comb.}}, 15(1):research paper r109, 53, 2008.

\bibitem{BDS}
N.~L. {Biggs}, R.~M. {Damerell}, and D.~A. {Sands}.
\newblock {Recursive families of graphs}.
\newblock {\em {J. Comb. Theory, Ser. B}}, 12:123--131, 1972.

\bibitem{Biggs1}
Norman {Biggs}.
\newblock {Chip-firing and the critical group of a graph}.
\newblock {\em {J. Algebr. Comb.}}, 9(1):25--45, 1999.

\bibitem{Biggs2}
Norman {Biggs}.
\newblock {The Tutte polynomial as a growth function}.
\newblock {\em {J. Algebr. Comb.}}, 10(2):115--133, 1999.

\bibitem{BLS}
Anders Bj\"orner, L\'aszl\'o Lov\'asz, and Peter~W. Shor.
\newblock Chip-firing games on graphs.
\newblock {\em {Eur. J. Comb.}}, 12(4):283--291, 1991.

\bibitem{BMM}
C.~{Brennan}, T.~{Mansour}, and E.~{Mphako-Banda}.
\newblock {Tutte polynomials of wheels via generating functions}.
\newblock {\em {Bull. Iran. Math. Soc.}}, 39(5):881--891, 2013.

\bibitem{Cal}
David Callan.
\newblock Some bijections for lattice paths, 2021.

\bibitem{CMS}
Yao-Ban {Chan}, Jean-Fran\c{c}ois {Marckert}, and Thomas {Selig}.
\newblock {A natural stochastic extension of the sandpile model on a graph}.
\newblock {\em {J. Comb. Theory, Ser. A}}, 120(7):1913--1928, 2013.

\bibitem{Chen}
Haiyan Chen and Bojan Mohar.
\newblock The sandpile group of polygon rings and twisted polygon rings, 2020.

\bibitem{CLB}
Robert {Cori} and Yvan {Le Borgne}.
\newblock {The sand-pile model and Tutte polynomials}.
\newblock {\em {Adv. Appl. Math.}}, 30(1-2):44--52, 2003.

\bibitem{CorPou}
Robert {Cori} and Dominique {Poulalhon}.
\newblock {Enumeration of \((p,q)\)-parking functions}.
\newblock {\em {Discrete Math.}}, 256(3):609--623, 2002.

\bibitem{Cori}
Robert {Cori} and Dominique {Rossin}.
\newblock {On the sandpile group of dual graphs}.
\newblock {\em {Eur. J. Comb.}}, 21(4):447--459, 2000.

\bibitem{ALB}
Michele D'Adderio and Yvan Le~Borgne.
\newblock The sandpile model on $ k\_m, n $ and the rank of its configurations.
\newblock {\em S{\'e}m. Loth. Comb.}, 77:Art--B77h, 2018.

\bibitem{DFF}
Arnaud {Dartois}, Francesca {Fiorenzi}, and Paolo {Francini}.
\newblock {Sandpile group on the graph \(\mathcal D_n\) of the dihedral group}.
\newblock {\em {Eur. J. Comb.}}, 24(7):815--824, 2003.

\bibitem{Dhar1}
Deepak {Dhar}.
\newblock {Self-organized critical state of sandpile automaton models}.
\newblock {\em {Phys. Rev. Lett.}}, 64(14):1613--1616, 1990.

\bibitem{Dhar}
Deepak Dhar.
\newblock Theoretical studies of self-organized criticality.
\newblock {\em Physica A}, 369(1):29--70, 2006.
\newblock Fundamental Problems in Statistical Physics.

\bibitem{Dukes}
Mark {Dukes}.
\newblock {The sandpile model on the complete split graph, Motzkin words, and
  tiered parking functions}.
\newblock {\em {J. Comb. Theory, Ser. A}}, 180:15, 2021.
\newblock Id/No 105418.

\bibitem{DLB}
Mark {Dukes} and Yvan {Le Borgne}.
\newblock {Parallelogram polyominoes, the sandpile model on a complete
  bipartite graph, and a \(q,t\)-Narayana polynomial}.
\newblock {\em {J. Comb. Theory, Ser. A}}, 120(4):816--842, 2013.

\bibitem{DSSS2}
Mark {Dukes}, Thomas {Selig}, Jason~P. {Smith}, and Einar {Steingr\'{\i}msson}.
\newblock {Permutation graphs and the abelian sandpile model, tiered trees and
  non-ambiguous binary trees}.
\newblock {\em {Electron. J. Comb.}}, 26(3):research paper p3.29, 25, 2019.

\bibitem{DSSS1}
Mark {Dukes}, Thomas {Selig}, Jason~P. {Smith}, and Einar {Steingr\'{\i}msson}.
\newblock {The abelian sandpile model on Ferrers graphs -- a classification of
  recurrent configurations}.
\newblock {\em {Eur. J. Comb.}}, 81:221--241, 2019.

\bibitem{GZ}
Curtis {Greene} and Thomas {Zaslavsky}.
\newblock {On the interpretation of Whitney numbers through arrangements of
  hyperplanes, zonotopes, non-Radon partitions, and orientations of graphs}.
\newblock {\em {Trans. Am. Math. Soc.}}, 280:97--126, 1983.

\bibitem{Het}
G\'abor {Hetyei}.
\newblock {Central Delannoy numbers and balanced Cohen-Macaulay complexes}.
\newblock {\em {Ann. Comb.}}, 10(4):443--462, 2006.

\bibitem{OEIS}
The OEIS~Foundation Inc.
\newblock {The On-line Encyclopedia of Integer Sequences}, 2022.

\bibitem{KW}
Seungki {Kim} and Yuntao {Wang}.
\newblock {A stochastic variant of the abelian sandpile model}.
\newblock {\em {J. Stat. Phys.}}, 178(3):711--724, 2020.

\bibitem{Kliv}
Caroline~J. {Klivans}.
\newblock {\em {The mathematics of chip-firing}}.
\newblock Boca Raton, FL: CRC Press, 2019.

\bibitem{Manna}
S~S Manna.
\newblock Two-state model of self-organized criticality.
\newblock {\em Journal of Physics A: Mathematical and General},
  24(7):L363--L369, apr 1991.

\bibitem{Mer}
Criel Merino.
\newblock {Chip firing and the Tutte polynomial}.
\newblock {\em {Ann. Comb.}}, 1(3):253--259, 1997.

\bibitem{Nunzi}
François Nunzi.
\newblock On the abelianity of the stochastic sandpile model, 2016.

\bibitem{Raza2}
Zahid {Raza}, Mohammed M.~M. {Jaradat}, Mohammed~S. {Bataineh}, and Faiz
  {Ullah}.
\newblock {On the sandpile model of modified wheels. II}.
\newblock {\em {Open Math.}}, 18:1531--1539, 2020.

\bibitem{Raza1}
Zahid {Raza}, Saleha {Tariq}, and M.~Tariq {Rahim}.
\newblock {Sandpile model on subdivided wheels \(W_{n,l}\)}.
\newblock {\em {Util. Math.}}, 105:291--302, 2017.

\bibitem{DharSad}
Tridib {Sadhu} and Deepak {Dhar}.
\newblock {Steady state of stochastic sandpile models}.
\newblock {\em {J. Stat. Phys.}}, 134(3):427--441, 2009.

\bibitem{Schulz}
Matthias {Schulz}.
\newblock {Minimal recurrent configurations of chip firing games and directed
  acyclic graphs}.
\newblock In {\em {Automata 2010. Selected papers based on the presentations at
  the 16th international workshop on cellular automata (CA) and discrete
  complex systems (DCS), Nancy, France, June 14--16, 2010}}, pages 111--124.
  Nancy: The Association. Discrete Mathematics \& Theoretical Computer Science
  (DMTCS), 2010.

\bibitem{SSS}
Thomas {Selig}, Jason~P. {Smith}, and Einar {Steingr\'{\i}msson}.
\newblock {EW-tableaux, Le-tableaux, tree-like tableaux and the abelian
  sandpile model}.
\newblock {\em {Electron. J. Comb.}}, 25(3):research paper p3.14, 32, 2018.

\bibitem{Sul}
Robert~A. {Sulanke}.
\newblock {Objects counted by the central Delannoy numbers}.
\newblock {\em {J. Integer Seq.}}, 6(1):art. 03.1.5, 19, 2003.

\bibitem{Zhou}
Yufang {Zhou} and Haiyan {Chen}.
\newblock {The sandpile group of a family of nearly complete graphs}.
\newblock {\em {Bull. Malays. Math. Sci. Soc. (2)}}, 44(2):625--637, 2021.

\end{thebibliography}

\end{document}